\documentclass{amsart}
\usepackage{blindtext}
\usepackage[french,english]{babel}
\usepackage{prodint}
\makeatletter
\renewcommand\part{%
   \if@noskipsec \leavevmode \fi
   \par
   \addvspace{4ex}%
   \@afterindentfalse
   \secdef\@part\@spart}

\def\@part[#1]#2{%
    \ifnum \c@secnumdepth >\m@ne
      \refstepcounter{part}%
      \addcontentsline{toc}{part}{\thepart\hspace{1em}#1}%
    \else
      \addcontentsline{toc}{part}{#1}%
    \fi
    {\parindent \z@ \raggedright
     \interlinepenalty \@M
     \normalfont
     \ifnum \c@secnumdepth >\m@ne
    \bfseries \partname\nobreakspace\thepart
       \par\nobreak
     \fi
      \bfseries #1%
     \par}%
    \nobreak
    \vskip 3ex
    \@afterheading}
\def\@spart#1{%
    {\parindent \z@ \raggedright
     \interlinepenalty \@M
     \normalfont
     \huge \bfseries #1\par}%
     \nobreak
     \vskip 3ex
     \@afterheading}
\makeatother

\usepackage{bm}

 \usepackage{upgreek}
 \usepackage{pifont}
   \usepackage{dsfont}

\usepackage{amsopn, amsthm, amsgen, amscd,amsmath, amssymb}
\DeclareMathAlphabet{\mathbbm}{U}{bbm}{m}{n}
\usepackage{color}
\usepackage{tikz}
  \usepackage{graphicx}
\usepackage{epstopdf}
\usepackage[small,nohug,heads = LaTeX]{diagrams} \diagramstyle[labelstyle=\scriptstyle]

\definecolor{CadetBlue}{cmyk}{0.62, 0.57, 0.23, 0 }
\definecolor{black}{cmyk}{1, 0.5, 0, 0 }
\definecolor{RedViolet}{cmyk}{0.07, 0.9, 0, 0.34 }
\definecolor{SeaGreen}{cmyk}{0.69, 0, 0.5, 0}

\DeclareMathAlphabet{\mathpzc}{OT1}{pzc}{m}{it}

\newcommand{\C}{\mathbb C}
\newcommand{\D}{\mathbb D}

\newcommand{\F}{\mathbb F}

\newcommand{\N}{\mathbb N}

\newcommand{\Z}{\mathbb Z}

\newcommand{\I}{\mathbb I}

\newcommand{\e}{\upvarepsilon}

\newcommand{\gotp}{\mathfrak{p}}
\newcommand{\gota}{\mathfrak{a}}
\newcommand{\gotm}{\mathfrak{m}}
\newcommand{\gotq}{\mathfrak{q}}
\newcommand{\gotn}{\mathfrak{n}}

\newcommand{\gotP}{\mathfrak{P}}

\newtheorem{theo}{Theorem}
\numberwithin{theo}{section}
\newtheorem{lemm}{Lemma}
\numberwithin{lemm}{section}
\newtheorem{prop}{Proposition}
\numberwithin{prop}{section}
\newtheorem{coro}{Corollary}
\numberwithin{coro}{section}

\numberwithin{rem}{section}

\theoremstyle{definition}

\numberwithin{defi}{section}

\numberwithin{axio}{section}

\theoremstyle{remark}

\numberwithin{nota}{section}

\numberwithin{exam}{section}

\newtheorem{note}{Note}
\numberwithin{note}{section}

\numberwithin{aside}{section}

\numberwithin{rema}{section}
\newtheorem*{clai}{Claim}

\title{Explicit Class Field Theory for Orders in Global Function Fields}
\author{L. Demangos}
\address{Xi'an Jiaotong - Liverpool University, Department of Mathematical Sciences, Mathematics Building Block B, 111 Ren'ai Road, Suzhou Dushu Lake Science
and Education Innovation District, Suzhou Industrial Park, Suzhou, Peoples Republic of China, 215123}
\email{Luca.Demangos@xjtlu.edu.cn}
\author{T.M. Gendron}
\address{Instituto de Matem\'{a}ticas -- Unidad Cuernavaca, Universidad
Nacional Aut\'{o}noma de M\'{e}xico, Av. Universidad S/N, C.P. 62210
Cuernavaca, Morelos, M\'{e}xico}
\email{tim@matcuer.unam.mx}
\date{\today}
\begin{document}
\vspace{2cm}

\begin{abstract}
This paper develops explicit class field theory for orders: of rank 1 in any global function field -- Hayes theory --  and of rank 2 in real quadratic function fields
-- Real Multiplication.
The essential ingredient in the development of the Hayes Theory is an orders version of Shimura's Main Theorem on Complex Multiplication, see \S \ref{s:D}.   The section on Real Multiplication
for orders uses values of the 
quantum modular invariant to generate  the Hilbert class field
of a rank 2 order contained in $\mathcal{O}_{K}$ = the integral closure of $\F_{q}[T]$. 
\end{abstract}
 \maketitle 
\tableofcontents
 \section*{Introduction}
Hayes theory, an explicit class field theory in positive characteristic, is  an outgrowth of Drinfeld's initial breakthrough \cite{Drin}, where \emph{elliptic modules}  -- originally introduced as a tool to provide a complete proof of the Langlands conjectures for ${\rm GL}_{2}$ over a global field in positive characteristic -- give rise to the natural analog of the classical theory of elliptic curves over a number field. 

While E.\ U.\ Gekeler introduced a \emph{modular invariant} for rank 2 Drinfeld modules with Complex Multiplication, in order to generate class fields of  \emph{imaginary quadratic global function fields}  \cite{Gek} in the style of the classical Theorem of Weber-Fueter, V. Drinfeld   \cite{Drin} and D. Hayes \cite{Hayes} explored alternative approaches to generate class fields over \emph{any} global function field $K$.  This was accomplished by replacing the Dedekind domain $\mathcal{O}_{K}$ by the Dedekind domain $A_{\infty}\subset\mathcal{O}_{K}$ of functions regular outside of a point $\infty\in\Upsigma_{K}$ = the curve associated to $K$.   Each such $A_{\infty}$ gives rise to its own collection of dedicated class fields, for which an explicit description can be made, and whose union gives the maximal abelian extension of $K$ unramified over $\infty$.

The key point is that the Dedekind domain $A_{\infty}$ embeds discretely into $\C_{\infty}$ (the function field complexes) as a non cocompact lattice, to which there is associated
an exponential that defines an analytic isomorphism 
\begin{align}\label{DWiso} (\C_{\infty},+)/A_{\infty} \cong (\C_{\infty},+)\end{align}which mirrors the behavior of the classical Weierstrass map in the new setting.  This isomorphism allows one to push forward
the analytic $A_{\infty}$-module
structure on  the left hand side of (\ref{DWiso}) to an algebraic $A_{\infty}$-module structure on the right hand side of (\ref{DWiso}), which is expressed by a representation of $A_{\infty}$
in the algebra of polynomials in the Frobenius $\uptau$,
\[   \uprho: A_{\infty}\longrightarrow \C_{\infty}\{ \uptau\}.\]
This is the Drinfeld module $\D$ associated to $A_{\infty}$; it gives rise to an analog of the theory of elliptic curves with Complex Multiplication.  

With this structure in hand, Hayes proved an analog of the Kronecker-Weber Theorem for {\it any} finite extension $K/\F_{q}(T)$.  In this setting, the Drinfeld module $\D$
is normalized by a suitable transcendental element $\upxi_{A_{\infty}}$ which plays the role of $\uppi$.  The 
Hilbert class field $H_{A_{\infty}}$ is generated by the coefficients of $\uprho (a)$ for any non-constant $a\in A_{\infty}$ and the
 ray class field $K_{\mathfrak{m}}$ associated to an ideal $\mathfrak{m}\subset A_{\infty}$ is generated over $H_{A_{\infty}}$ by $\D[\mathfrak{m}]$ =  the (cyclic)  module of $\mathfrak{m}$-torsion.



It is of interest to extend Hayes' Theory to orders $R\subset A_{\infty}$:  indeed, Hayes himself initiated such a study.  In  \cite{Hayes0}, Hayes developed the theory of Drinfeld modules
over $R$, proved the generation result for $H_{R}$ but only gave ray class field generation in the case $R=A_{\infty}$.  

The main aim of this paper is to 
conclude Hayes' original work by proving the class field generation result for arbitrary $R\subset A_{\infty}$.   In doing so, we revamp Hayes' original account by introducing
 the modular invariant $j(\mathfrak{a})$ of an {\it invertible} $R$-ideal $\mathfrak{a}$ and show that $H_{R}= K(j(\mathfrak{a}))$, which makes the theory closer in spirit with that of elliptic curves, via
 the Weber-Feuter Theorem.  The main tool we use in the ray class field generation result is a version of Shimura's Main Theorem valid in this setting, proved here in \S \ref{s:D}.
 
In \S \ref{HigherHCF}, as an application, we prove a generalization of the main theorem of \cite{DGIII} to the order context: if $K/\F_{q}(T)$ is {\it real} and quadratic,
$f\in\mathcal{O}_{K}$ a fundamental unit and $\mathcal{O}_{f}=\F_{q}[T,f,f^{-1}]\subset \mathcal{O}_{K}$ then
\[ H_{\mathcal{O}_{f}}  = K \left( \prod_{\upalpha\in j^{\rm qt}(f)} \upalpha \right).\]
This result, as well as the Hayes Theory for orders, are essential to the generation of ray class fields over rank 2 rings,  where the theory we develop only works for orders of the form
$\mathcal{O}_{f}$ and not $\mathcal{O}_{K}$, see \cite{DGIV}.

 \section{Orders in Function Fields}\label{s:A}

In this section we state and if necessary, prove, some general facts about orders in function fields.  In what follows, we fix a curve $\Upsigma_{K}$ over $\F_{q}$ with ${\rm Rat}(\Upsigma_{K})= K$,  a point $\infty\in\Upsigma_{K}$ and the Dedekind domain 
\[  A= A_{\infty}= \{\text{elements of $K$ regular outside of $\infty$}\} .\]

Let $R\subset A$ be an {\bf {\em order}}: by definition a subring with $1$ whose field of quotients is $K$.
In what follows ${\rm ord}_{\infty}(f) $ is the order of $f\in K$ at $\infty$.  By definition, for all $f\in A$, ${\rm ord}_{\infty}(f) \leq 0$.

\begin{lemm}\label{EventEveryOrder} There exists $N\in\N$ such that for all $k\geq N$, there exists $g\in R$ with ${\rm ord}_{\infty}(g) =-k$.
\end{lemm}

\begin{proof}
The field $K= {\rm Frac}(R)$ has an element of order $-1$ at $\infty$, therefore there exists $r,s\in R$ having orders $-m$, $-m-1$ at $\infty$.  
The set of valuations of $r^{n}s^{n'}$, $n,n'\in\N$ is the additive submonoid of $\Z$ generated by $m,m-1$.  Since $(m,m-1)=1$, the 
result follows. 
\end{proof}

\begin{coro}\label{FiniteModule} $(A,+)/(R,+)$ is a finite $R$-module, and in particular, $A$ is a finitely generated $R$-module. More generally, $R/\mathfrak{a}$ is finite for any $R$-ideal $\mathfrak{a}\subset R$.
\end{coro}

\begin{proof}   By Lemma \ref{EventEveryOrder}, any coset of the quotient $R$-module $(A,+)/(R,+)$ has a representative of order $\geq -N$, where $N$ is as in the Lemma.  But the set of elements of $A$ having
order bounded below by $-N$ is finite, since $A$ is an $\F_{q}$ algebra whose 
 elements are those elements of $K$ regular outside of $\infty$.  Thus $(A,+)/(R,+)$ is finite.   The second statement follows from an identical argument, since any ideal $\mathfrak{a}\subset R$ also has the property
 described in the statement of Lemma \ref{EventEveryOrder}, of having elements of all orders less than some lower bound.
\end{proof}

\begin{coro}\label{AintegraloverR} $A$ is integral over $R$.  In particular, $\overline{R}=$ integral closure of $R$ = $A$.
\end{coro}

\begin{proof} Since $(A,+)/(R,+)$ is finite, given $f\in A$, we may write $f^{m}-f^{n} =r\in R$
for some $m,n\in\N$ hence $f$ satisfies $P(X)= X^{m}-X^{n}-r\in R[X]$.  Thus $A\subset \overline{R}$.  But since $R\subset A$ and $A$ is Dedekind, $\overline{R}=A$.
\end{proof}

\begin{coro}\label{Dim1Noeth} $R$ is a {\rm 1}-dimensional Noetherian domain.
\end{coro}

\begin{proof}  By Corollary \ref{FiniteModule}, $A$ is a finitely generated $R$-module, and by Proposition 7.8
of \cite{AM}, $R$ is a finitely generated $\F_{q}$ algebra, hence Noetherian.   By  Lemma \ref{EventEveryOrder},  every principal ideal $aR$ has elements of all orders $-k\leq -N +{\rm ord}_{\infty}(a)$ 
hence $R/aR$ is a finite $R$-module.  In particular, if $\mathfrak{p}\subset R$ is a prime ideal, $R/\mathfrak{p}$ is a finite
integral domain, hence a field.  Thus $\mathfrak{p}$ is maximal, so the Krull dimension is 1.
\end{proof}

 \begin{prop}\label{contractionprop}  Every prime ideal of $ R$ is the contraction of a prime ideal in $A$.
 \end{prop}
 
 \begin{proof} 
$A$ is integral over $R$ by  Corollary \ref{AintegraloverR}, so  by Proposition 9 on page 9 of \cite{Lang},
 for any prime $\mathfrak{p}\subset R$, $\mathfrak{p}A\not= A$, hence there is a prime $\widetilde{\mathfrak{p}}\supset \mathfrak{p}A$.  By Corollary \ref{Dim1Noeth}, $R$ has Krull dimension 1, hence this prime contracts to $\mathfrak{p}$.
 \end{proof}
 
 We recall that the {\bf {\em conductor}} \[ \mathfrak{c}\subset R \]is the largest $A$-ideal contained in $R$.  The conductor is always non-zero.
 
 \begin{theo}\label{bijectionprimetoc} The contraction map 
 \begin{align}\label{contractionbijection}  A\supset \mathfrak{a}\longmapsto \mathfrak{a}_{R}:= \mathfrak{a}\cap R\subset R \end{align} 
 induces  a bijective, multiplicative correspondence between ideals of $A$ prime to $\mathfrak{c}$ and ideals of $R$ prime to $\mathfrak{c}$.  The inverse correspondence of {\rm (\ref{contractionbijection})} is given by the expansion,
 $\mathfrak{a}_{R}\mapsto  A\mathfrak{a}_{R}$.   The bijection {\rm (\ref{contractionbijection})} induces a bijection between prime ideals prime to the conductor.
\end{theo}

\begin{proof}  The proof is formally the same as that for orders in rings of integers in number fields, for the latter see Theorem 3.8 of \cite{Conrad1}.
\end{proof}

For any fractional $R$-ideal $\mathfrak{a}$, define
\begin{align}\label{defstarop}  \mathfrak{a}^{\ast} = \{ \upbeta\in K\; :\;\; \upbeta\mathfrak{a}\subset R\}. \end{align}
Thus $ \mathfrak{a}^{\ast} \mathfrak{a}\subset R$;
if $ \mathfrak{a}^{\ast} \mathfrak{a}=R$, we say that $\mathfrak{a}$ is invertible and write $\mathfrak{a}^{-1}= \mathfrak{a}^{\ast}$.
All principal ideals of $R$ are invertible, and
if $R$ is Dedekind, every ideal is invertible.   

\begin{lemm}  Let $\mathfrak{b}\subset R$ be an ideal and
\[ R^{ \mathfrak{b}}:=\{ x\in K\; : \;\; x\mathfrak{b}\subset\mathfrak{b}\}. \]  Then $R\subset R^{ \mathfrak{b}}\subset A$ and
if $\mathfrak{b}$ is prime to the conductor,  $ R^{ \mathfrak{b}} =R$.  
\end{lemm}

\begin{proof} $R^{ \mathfrak{b}}$ is a commutative ring with $1$ and clearly 
\[ \mathfrak{b}^{\ast}\supset R^{ \mathfrak{b}}\supset R,\] so $R^{ \mathfrak{b}}$ is an $R$-module.  We claim that the quotient $R$-module $(R^{ \mathfrak{b}},+)/(R,+)$ is finite: this will follow if
 $( \mathfrak{b}^{\ast} ,+)/ (R,+)$ is finite.  However, for any $y\in \mathfrak{b}$,  $( \mathfrak{b}^{\ast} ,+)/ (R,+)$
is contained in $(y^{-1}R,+)/(R,+) \cong R/yR$, which is finite by Corollary \ref{FiniteModule}.  Given $x\in R^{ \mathfrak{b}}$, all of the powers $x^{n}$, $n\geq 0$, define elements of $(R^{ \mathfrak{b}},+)/(R,+)$, so
there exist exponents $m,n$ with $x^{n}-x^{m}=r\in R$.  Thus $x$ is integral over $R$, hence $x\in A$ by Corollary \ref{AintegraloverR}.  If $\mathfrak{b}$ is prime to $\mathfrak{c}$, let $b\in \mathfrak{b}$, $c\in\mathfrak{c}$
with $b+c=1$.  Then $x = xb+xc\in R$.
\end{proof}



\begin{theo}  Let $0\not= \mathfrak{p}\subset R$ be prime.  Then $\mathfrak{p}$ is invertible $\Leftrightarrow$ $R^{\mathfrak{p}} =R$.
\end{theo}

\begin{proof} Formally the same as the proof of Theorem 3.4 of \cite{Conrad1}.
\end{proof}

\begin{theo}\label{inverttheo} If $\mathfrak{b}\subset R$ is prime to $\mathfrak{c}$, then $\mathfrak{b}$ may be written as a product of invertible primes.  In particular, $\mathfrak{b}$ is invertible.
\end{theo}

\begin{proof} Formally the same as the proof of Theorem 3.6 of \cite{Conrad1}.
\end{proof}

The class group of $R$ is by definition
\[  {\sf Cl}_{R} = \{ \text{invertible fractional $R$ ideals}\}/ \{ \text{principal $R$ ideals}\} .\]
The proof of the following is due to W. Sawin.


\begin{lemm}\label{relprimerep} Let $\mathfrak{n}\subset R$ be an ideal.  Then every class $[\mathfrak{a}]\in {\sf Cl}_{R}$ contains a representative prime to $\mathfrak{n}$.
\end{lemm}

\begin{proof} When $R=A$ is Dedekind, this is an elementary consequence of the Chinese remainder theorem, see \cite{Conrad1}, Lemma 5.1.   Otherwise, let $\mathfrak{a}$ be an invertible fractional $R$-ideal,
 $\upalpha\in \mathfrak{a}^{-1}$, so that $\upalpha\mathfrak{a}\subset R$ is an integral representative of $[\mathfrak{a}]$.  For any prime $\mathfrak{p}\subset R$, then, we have 
 \begin{align*}    \quad\quad\quad (\upalpha\mathfrak{a}  ,\mathfrak{p}) =1 
 \Longleftrightarrow \upalpha\mathfrak{a} \not\subset \mathfrak{p} \Longleftrightarrow \exists \upbeta\in\mathfrak{a}, \; \upalpha\upbeta\not\in\mathfrak{p}. &
  \quad\quad\quad   (\dag_{\mathfrak{p}})  \quad\quad\quad 
  \end{align*}
  In particular, $(\upalpha\mathfrak{a}  ,\mathfrak{n}) =1 $ if and only if $ (\dag_{\mathfrak{p}}) $ holds for all $\mathfrak{p}|\mathfrak{n}$.  The satisfaction of $ (\dag_{\mathfrak{p}}) $ for a particular $\upalpha$ only
  depends on the coset $\upalpha \mod \mathfrak{p}\mathfrak{a}^{-1}$ and, moreover,  the quotient module $\mathfrak{a}^{-1}/ \mathfrak{p}\mathfrak{a}^{-1}$ contains a class defined by $\upalpha_{\mathfrak{p}}$ satisfying
  $ (\dag_{\mathfrak{p}}) $, since $\mathfrak{a}\mathfrak{a}^{-1}\supsetneq \mathfrak{p}$.  By the Chinese remainder theorem for modules, the map 
  \[   \mathfrak{a}^{-1}\longrightarrow \prod_{\mathfrak{p}|\mathfrak{n}} \mathfrak{a}^{-1}/ \mathfrak{p}\mathfrak{a}^{-1} \]
  is surjective, hence there exists $\upalpha\in\mathfrak{a}^{-1}$ satisfying $ (\dag_{\mathfrak{p}}) $ for all $\mathfrak{p}|\mathfrak{n}$.
  
\end{proof}

\begin{coro}\label{surjectionofclassgroups}  The expansion map 
\[ R\supset\mathfrak{a}\longmapsto \mathfrak{a}A \subset A\]
induces an epimorphism
\[ {\sf Cl}_{R}\twoheadrightarrow {\sf Cl}_{A}.\]  
\end{coro}

\begin{proof} Let $[\mathfrak{A}]\in {\sf Cl}_{A}$.  By Lemma \ref{relprimerep}, we may assume $\mathfrak{A}$ is prime to $\mathfrak{c}$.  Then by Theorem \ref{bijectionprimetoc}, $\mathfrak{a}=\mathfrak{A}\cap R$
is also prime to the conductor, hence invertible by Theorem \ref{inverttheo}, and so defines an element $[\mathfrak{a}]\in {\sf Cl}_{R}$ mapping onto $[\mathfrak{A}]$.
\end{proof}

\begin{coro}\label{PrimeToConductor} An ideal in $R$ that is  prime to the conductor has unique factorization into prime
ideals which are prime to the conductor.   In particular, a prime ideal prime to the conductor is invertible. All but finitely many prime ideals are prime to the conductor, hence all 
but finitely many prime ideals are invertible.\end{coro}

\begin{proof} Again, the proof here is formally the same as that for orders in number fields, see Corollary 3.11 of \cite{Conrad1}.
\end{proof}

\begin{theo}\label{invertiffprimetoc} A prime ideal $\mathfrak{p}\subset R$ is invertible $\Longleftrightarrow$ $\mathfrak{p}$ is prime to the conductor.
\end{theo}

\begin{proof} $\Longleftarrow$ By Corollary \ref{PrimeToConductor}.  

\noindent $\Longrightarrow$ The proof is formally the same as that of Theorem 6.1 of  \cite{Conrad1}; in this connection, we remark that the proof requires that $A$ is a finitely generated $R$-module, which
is a consequence of Corollary \ref{FiniteModule}.
\end{proof}



\begin{theo}\label{2gentheo} Every invertible ideal $\mathfrak{a}\subset R$ has the {\rm 2}-generator property: there exists $\upalpha, \upbeta\in R$ with 
\[  \mathfrak{a} = (\upalpha, \upbeta). \]
\end{theo}

\begin{proof}  Every ideal $\mathfrak{a}\subset R$, invertible or not, is contained in finitely many prime ideals.  Indeed, by Proposition \ref{contractionprop}, every prime
ideal of $R$ is of the form $\mathfrak{p}_{R}= \mathfrak{p}\cap R$ for $\mathfrak{p}\subset A$ prime.  Thus if $\mathfrak{a}$ is contained in infinitely many distinct primes $\mathfrak{p}_{i,R}$
we would have
\[ \mathfrak{a} \subset \bigcap_{i=1}^{\infty} \mathfrak{p}_{i,R}\subset  \bigcap_{i=1}^{\infty} \mathfrak{p}_{i}  = (0).\]
By Corollary 1 of \cite{GilmerHeinzer}, any invertible ideal in a commutative ring contained in a finite number of maximal ideals has the 2-generator property.
\end{proof}

\begin{note}\label{mingennote} By Theorem 3 of \cite{GilmerHeinzer}, given any $0\not=a\in\mathfrak{a}\subset R$ an invertible ideal, there exists $b\in \mathfrak{a}$ with $\mathfrak{a}=(a,b)$.
\end{note}

 
Let  ${\sf I}_{R}$ be the monoid of fractional $R$-ideals, ${\sf I}^{\ast}_{R}$ the subgroup of invertible ideals, ${\sf P}_{R}$ the subgroup of principal ideals, so that
$ {\sf Cl}_{R}  ={\sf I}^{\ast}_{R}/{\sf P}_{R} .$
Denote by 
\[ {\sf I}_{R}^{\mathfrak{c}}<{\sf I}^{\ast}_{R}\] the subgroup of $R$-ideals prime to the conductor and ${\sf P}^{\mathfrak{c}}_{R}<{\sf P}_{R}$ the subgroup of
principal $R$-ideals prime to the conductor.  Similarly, let ${\sf I}_{A}^{\mathfrak{c}}$ be the group of  $A$-ideals  prime to the conductor and let \[ {\sf P}^{\mathfrak{c}}_{A} = \langle rA \; | \;\; r\in R,\; (r,\mathfrak{c}) =1\rangle \]
be the group of principal
$A$-ideals generated by an element of $R$ which is prime to the conductor.

\begin{theo}\label{CGIsos} We have the following isomorphisms
\[   {\sf Cl}_{R}\cong  {\sf I}_{R}^{\mathfrak{c}}/{\sf P}_{R}^{\mathfrak{c}}  \cong  {\sf I}_{A}^{\mathfrak{c}}/{\sf P}_{A}^{\mathfrak{c}} . \]
\end{theo}

\begin{proof} See \cite{Hayes0}.
\end{proof}

We close this section with a discussion of localizations and completions.  If $\mathfrak{p}_{R}\subset R$ is a prime ideal, by Proposition \ref{contractionprop}, it is the contraction of a 
prime $\mathfrak{p}\subset A$.  Denote the localization with respect to $\mathfrak{p}_{R}$ by 
\[ R_{( \mathfrak{p}_{R} )} = RS^{-1}_{\mathfrak{p}_{R}}, \quad S_{\mathfrak{p}_{R}}= R\setminus \mathfrak{p}_{R} .\]
Since $ S_{\mathfrak{p}_{R}}\subset  S_{\mathfrak{p}} = A\setminus \mathfrak{p}$, there is a canonical inclusion
\[  R_{( \mathfrak{p}_{R} )} \hookrightarrow A_{( \mathfrak{p} )}.  \]

\begin{prop}\label{WhenLocalEqual} If $\mathfrak{p}\nmid \mathfrak{c}$ then $R_{( \mathfrak{p}_{R} )} = A_{( \mathfrak{p} )}$.
\end{prop}

\begin{proof} By hypothesis, there exists $c\in\mathfrak{c}\setminus\mathfrak{p}\subset  S_{\mathfrak{p}}$.  Then for all $\upalpha/s\in   A_{( \mathfrak{p} )}$, $\upalpha/s=c\upalpha/cs\in R_{( \mathfrak{p}_{R} )}$.
\end{proof}

The completion of $R_{( \mathfrak{p}_{R} )}$ with respect to its maximal ideal $\mathfrak{p}_{R} R_{( \mathfrak{p}_{R} )}$ is denoted $R_{\mathfrak{p}_{R}}$.

There is another way to complete $R_{( \mathfrak{p}_{R} )} $: namely,  with respect to its embedding into the completion $A_{\mathfrak{p}}\subset K_{\mathfrak{p}}$  of $A$.  Denote
this completion $R_{\mathfrak{p}}$.

\begin{prop}\label{CompletionProp} The completion $R_{\mathfrak{p}_{R}}$ is canonically isomorphic, as a topological ring, to $R_{\mathfrak{p}}$.
\end{prop} 

\begin{proof}  The canonical inclusion of localizations $R_{( \mathfrak{p}_{R} )}\hookrightarrow A_{(\mathfrak{p})}$ induces a continuous inclusion of completions 
$R_{\mathfrak{p}_{R} }\hookrightarrow  R_{\mathfrak{p}}\subset  A_{\mathfrak{p}}$, since $ \mathfrak{p}_{R} R_{( \mathfrak{p}_{R} )}\subset \mathfrak{p}A_{(\mathfrak{p})}$, the maximal ideal of $A_{(\mathfrak{p})}$.  To see that
this map is bicontinuous, we must show that the two completion topologies are compatible: that is,  there is a power of $\mathfrak{p}A_{(\mathfrak{p})}$ contained in $ \mathfrak{p}_{R} R_{( \mathfrak{p}_{R} )} $.  But  $R_{( \mathfrak{p}_{R} )}$ is an order in 
the Dedekind domain $A_{(\mathfrak{p})}$; denote the corresponding conductor
by $\mathfrak{c}_{(\mathfrak{p})} $. 
Since
$ A_{(\mathfrak{p})}$ is Dedekind,  $\mathfrak{c}_{(\mathfrak{p})} \subset \mathfrak{p}_{R} R_{(\mathfrak{p}_{R})}$ is a power of $\mathfrak{p}A_{(\mathfrak{p})}$.   That is, convergence to zero with respect to 
either topology is the same, and the completions with respect to either topology produce canonically isomorphic rings.
\end{proof}

In view of  Proposition \ref{CompletionProp}, we may write unambiguously $R_{\mathfrak{p}}$ for the completion of $R_{(\mathfrak{p}_{R})}$ with respect to either topology.

\begin{coro}\label{CompletionIntersectCoro}  $K\cap R_{\mathfrak{p}} = R_{( \mathfrak{p}_{R} )}$.
\end{coro}

\begin{proof} The inclusion $R_{( \mathfrak{p}_{R} )}\subset K\cap R_{\mathfrak{p}} $ is clear. Let $\upalpha\in K\cap R_{\mathfrak{p}}$.  Then there exists a sequence $x_{i}/s_{i}\in R_{( \mathfrak{p}_{R} )}$ converging to $\upalpha$.  By definition
of the completion topology, for any $n$ and $i$ large, $\upalpha - x_{i}/s_{i} \in \mathfrak{p}_{R}^{n}R_{( \mathfrak{p}_{R} )}$.  But this implies {\it a fortiori} that
eventually $\upalpha - x_{i}/s_{i} \in R_{( \mathfrak{p}_{R} )}$, which gives  $\upalpha\in R_{( \mathfrak{p}_{R} )}$.  
\end{proof}

\begin{prop} The completion of $R\subset R_{(\mathfrak{p})}$ (with respect to either completion topology) is $R_{\mathfrak{p}}$. 
\end{prop}

\begin{proof} It is enough to show that $R\subset R_{(\mathfrak{p})}$ is dense in either topology, and to show this, it is enough to show that given $s\in S_{\mathfrak{p}_{R}}$, there
exists a sequence $( y_{i} )_{i=1}^{\infty}\subset R$ converging to $1/s$, or equivalently, $y_{i}s\rightarrow 1$.  But $(s, \mathfrak{p}_{R}^{i})=1$ for all $i$, so choose $y_{i}$ such that $y_{i}s + z_{i}\uprho_{i} =1$
where $\uprho_{i} \in \mathfrak{p}_{R}^{i}$ and $z_{i}\in R$: this gives the desired sequence.
\end{proof}

\section{Drinfeld and Hayes Modules over Rank 1 Orders}\label{s:C}


The theory of Drinfeld modules over an order was introduced in \cite{Hayes0}, however, this paper was written before Hayes had arrived at the final
form of his theory (c.f.\ \cite{Hayes}).  In particular, in  \cite{Hayes0}, Hayes did not introduce the notion of sign
normalization and only gave an explicit generation result for the Hilbert class field of an order.  Regarding ray class fields, Hayes elected to restrict to the case when the order is Dedekind.
The purpose of this section then is to introduce the reader to Hayes' foundational work, updating it according to the requirements of the present paper.



 We continue to use notation established in section \S \ref{s:A}.
We will use lower case gothic letters such as $\mathfrak{a}$ to denote ideals in $A$; ideals in $R$ will be denoted by lower case gothic letters with the subindex $R$ e.g.\ $\mathfrak{a}_{R}$.  
If $\mathfrak{a}_{R}\subset R$ happens to be a contraction of an ideal in $A$, the latter ideal will be denoted $\mathfrak{a}$. 
  
 

The completions $K_{\infty}\subset \C_{\infty}$ have their usual meanings.  The Frobenius automorphism acting on $\C_{\infty}$ is denoted $\uptau (x)=x^{q}$, and for any subfield $L\subset \C_{\infty}$, 
\[ L\{ \uptau \}\] is the noncommutative algebra of additive polynomials in $\uptau$, with product given by composition.
Let \[ \upiota: A\rightarrow L\]be a ring homomorphism, making $L$ an $A$-field: the characteristic is defined to be $\wp = {\rm Ker}(\upiota )$.  If $\wp=(0)$ we say that the characteristic is generic.  A rank 1 Drinfeld $R$-module defined over $L$
 \[   \D = (\C_{\infty}, \uprho ), \]
consists of an $\F_{q}$-algebra homomorphism
\[ \uprho : R\longrightarrow L\{ \uptau\} \subset \C_{\infty}\{ \uptau\} \]
in which each $\uprho_{a}:= \uprho (a)$ has the form
 \[ \uprho_{a}(\uptau ) = \upiota (a)\uptau^{0} + a_{1}\uptau +\dots +a_{d}\uptau^{d},\quad d=\deg (a),\quad a_{1},\dots ,a_{d}\in L,\;\; a_{d}\not=0 .\]
Note that $a_{0}= \upiota (a)$ implies that
\[ \upiota = D\circ \uprho,\]
where $D=d/dx$ is the derivative with respect to $x$.  

We have the following analytical version of rank 1 Drinfeld $R$-modules. See \S 4 of \cite{Hayes0}. By a rank 1 $R$-lattice is meant a discrete rank 1 $R$-submodule $\Uplambda$ of $\C_{\infty}$.  
By Dirichlet's Unit Theorem for global fields (see Theorem 3.3 on page 102 of \cite{cohn}), $A^{\times}= \mathbb{F}_{q}^{\times}$, which implies as well that $R^{\times}=\F_{q}^{\times}$ (since ${\rm Frac}(R)=K$).  In particular,
$R$ and all of its fractional ideals are rank 1 $R$-lattices in $\C_{\infty}$, and any rank 1 $R$-lattice $\Uplambda$ is of the form $\upxi\mathfrak{a}_{R}$ for $\upxi\in \C_{\infty}$ and $\mathfrak{a}_{R}$ a fractional $R$-ideal.   To  
an $R$-lattice $\Uplambda$ we may associate the exponential function
\[ e_{\Uplambda}:\C_{\infty}\longrightarrow \C_{\infty},\quad e_{\Uplambda}(z) = z\prod_{0\not= \uplambda\in\Uplambda}\left(1- \frac{z}{\uplambda}\right). \]
The exponential function in turn defines a unique Drinfeld $R$-module $\D=(\C_{\infty},\uprho)$, isomorphic as an $R$-module to $\C_{\infty}/\Uplambda$ via $e_{\Uplambda}$ i.e.\
\begin{align}\label{fundform}  e_{\Uplambda}(az) = \uprho_{a} (e_{\Uplambda}(z)),\quad \text{for all }a\in A.\end{align}
See page 188 of \cite{Hayes0}.
Every rank 1 Drinfeld $R$-module $\D=(\C_{\infty},\uprho )$ in generic characteristic may be obtained in this way. 
 See \cite{Hayes0}, Theorem 5.9.   We denote by $\Uplambda_{\uprho}$ and $e_{\uprho}$ the corresponding
lattice and exponential map. 

At this point, we depart slightly from \cite{Hayes0} and introduce the notion of an {\bf {\em invertible rank 1 Drinfeld $R$-module}}:
a generic characteristic Drinfeld $R$-module with lattice of the form $\upxi \mathfrak{a}_{R}$, for $\mathfrak{a}_{R}\subset R$ an invertible ideal.   For
$\D$ an invertible rank 1 Drinfeld module,  \[ {\rm End}(\D)\cong {\rm End}(\mathfrak{a}_{R}) = \{ \upalpha\in K\; : \;\; \upalpha \mathfrak{a}_{R}\subset \mathfrak{a}_{R} \} =  R,\]
where the last equality follows immediately from the invertibility of $\mathfrak{a}_{R}$.

\vspace{5mm}

\noindent \fbox{\bf Important Assumption}  In what follows, all rank 1 Drinfeld $R$-modules will be assumed invertible.  

\vspace{5mm}

The Assumption appears to be a hypothesis necessary in the proof of the analog of Shimura's Main Theorem, see section \S \ref{s:D}.

Denote by $\F_{\infty}\supset \F_{q}$ the field of constants of $K_{\infty}$, $d_{\infty} = [ \F_{\infty}: \F_{q}]$. We fix a sign function: a homomorphism
\[ {\rm sgn}:K_{\infty}^{\times}\longrightarrow \F_{\infty}^{\times}\]
which is the identity on $\F_{\infty}^{\times}$ and trivial on the 1-units.  There are exactly \[ \# \F_{\infty}^{\times}=q^{d_{\infty}}-1\] sign functions.  For any $\upsigma\in {\rm Gal}(\F_{\infty}/\F_{q})$,
a twisted sign function is a homomorphism of the form $\upsigma\circ {\rm sgn}$. 

For $\D =(\C_{\infty},\uprho )$ a rank 1 Drinfeld module, the leading coefficient of $\uprho_{a}$ is denoted \[ \upmu_{\uprho}(a):= a_{d} ,\]
where $d$ is the degree of $a$ at $\infty$.
  By
  Lemma \ref{EventEveryOrder},  $R$ contains functions whose pole at $\infty$ is of order $-n$, for all $n$ sufficiently large.
Then following the standard arguments (c.f.\ \cite{Thak}, page 68), the map $R\ni a\mapsto  \upmu_{\uprho}(a)$ may be extended to $K_{\infty}^{\times}$.   The restriction of $\upmu_{\uprho}$ to $\F_{\infty}$
gives an automorphism $i_{\uprho}$ fixing $\F_{q}$ i.e.\ an element of ${\rm Gal}(\F_{\infty}/ \F_{q})$.

We say that $\D$ is {\bf {\em sign normalized}} or a {\bf {\em Hayes module}} if $\upmu_{\uprho}$ is a twisting
of the sign function.   See \cite{Hayes} for a discussion of this notion in the case where $R$ is Dedekind. 

\begin{prop}
Let $h_{R}$ be the class number of $R$. Then there are exactly $h_{R}$ isomorphism classes of rank 1 Drinfeld modules over $\C_{\infty}$.  
\end{prop}

\begin{proof}
When $R=A$ is Dedekind, this is the content of \cite{Hayes0}, Corollary 5.13, page 195.  For general $R$, the proof is essentially the same: since we are assuming that our Drinfeld
modules are uniformized by lattices homothetic to invertible ideals, and isomorphic Drinfeld modules have homothetic lattices, the corresponding ideals must define the same class in ${\sf Cl}_{R}$.
\end{proof}

\begin{prop} Every isomorphism class $[\D]$ of Drinfeld $R$-module contains exactly $(q^{d_{\infty}}-1)/(q-1)$ Hayes modules.
\end{prop}

\begin{proof}  Let $\D =  (\C_{\infty},\uprho )$ be a representative of the isomorphism class and let $\uppi$ be a uniformizer at the prime $\infty$ with ${\rm sgn}(\uppi )=1$.  Let $z\in\C_{\infty}$ satisfy $z^{1-q^{d_{\infty}}}= \upmu_{\uprho} (\uppi^{-1})$.
Then the Drinfeld module $\D_{0}$ defined by $\uprho_{0} = z\uprho z^{-1}$ satisfies $\upmu_{\uprho_{0}} (\uppi )=1$.  From this it follows that for any $x= \sum_{i\geq n} c_{i} \uppi^{n} \in K_{\infty} = \F_{\infty} ((\uppi))$,
$c_{i}\in\F_{\infty}$, 
\[ \upmu_{\uprho_{0}} (x) = \upmu_{\uprho_{0}} (c_{n}) = \upmu_{\uprho_{0}} ({\rm sgn}(x) ) = i_{\uprho}( {\rm sgn}(x)). \]
Thus the class $[\D]$ contains the Hayes module $\D_{0}$.  The rest of the argument is the same as that in the case of a Dedekind domain, see page 69 of \cite{Thak}.
\end{proof}

 Each Hayes module $\D=(\C_{\infty},\uprho )$ is thus associated to a class of invertible ideal $\mathfrak{a}_{R}$ which we assume is contained in $R$. Then, for each Hayes module $\D$
 in the associated class, there is a transcendental element
$\upxi_{\uprho}\in \C_{\infty}$ so that the Drinfeld module of the lattice $\Uplambda_{\uprho}:=\upxi_{\uprho} \mathfrak{a}_{R}$ is $\D$.


We define the Hilbert class field $H_{R}$ associated to $R$ (c.f.\ \cite{Hayes0}, Theorem 8.10) as the minimal field of definition of any (invertible) Drinfeld $R$-module.  To state the results of Hayes Theory, which we extend
to $R$,  we also require the narrow version of the Hilbert class field.  
The narrow class group of $R$ may be identified with the quotient
\[ {\sf Cl}^{1}_{R} :={\sf I}_{R}^{\mathfrak{c}}/ {\sf P}^{\mathfrak{c},1}_{R}\] 
where ${\sf I}_{R}^{\mathfrak{c}}$ is as before the group of fractional ideals of $R$ prime to the conductor and  $ {\sf P}^{\mathfrak{c},1}_{R}$ is the subgroup of principal ideals prime to the conductor that are generated by a ${\rm sgn}$ one element.  Then 
\[ h_{R}^{1}:=\# {\sf Cl}^{1}_{R} =h_{R} (q^{d_{\infty}}-1)/(q-1) \]
where $h_{R}$ is the class number of $R$.  

We now translate some of these constructions into the id\`{e}lic language\footnote{We continue the convention of indexing factors of the id\`{e}les using primes in $A$ along with $\infty$.}, see for example \cite{Thak}, page 80.
If $\mathfrak{p}\subset A$ is prime, by $R_{\mathfrak{p}}$ we mean the completion of $R$ in $K_{\mathfrak{p}}$.
Then, if we let
\[ U_{R}^{1} :=\{ s\in \I_{K}|\; \text{for all }\mathfrak{p}\subset A, \;s_{\mathfrak{p}}\in R_{\mathfrak{p}}^{\times} \text{ and } {\rm sgn}(s_{\infty})=1\} ,\]
and if $\uppi_{\infty}$ is a uniformizer of $K_{\infty}$ with ${\rm sgn}(\uppi_{\infty})=1$, then the narrow Hilbert class field 
\[ H^{1}_{R}\]
is defined to be the class field corresponding to the group \[ \I^{1}_{R}:=K^{\times}\cdot \uppi_{\infty}^{\Z} \cdot U_{R}^{1}\subset\I_{K}.\]  Thus, 
 Artin reciprocity gives an isomorphism
\[  [\cdot, K]: \I_{K}/\I^{1}_{R}\longrightarrow {\rm Gal}(H_{R}^{1}/K), \]
which induces, on the level of ideal classes, an isomorphism
\[ {\sf Cl}^{1}_{R}\longrightarrow {\rm Gal}(H_{R}^{1}/K) ,\quad \mathfrak{a}_{R}\longmapsto \upsigma_{\mathfrak{a}_{R}}.\]
When $d_{\infty}=1$, i.e., $\F_{\infty}=\F_{q}$,
then $H_{R}^{1}=H_{R}$.

If $\mathfrak{a}_{R}\subset R$ is a (not necessarily invertible) ideal and $\D=(\C_{\infty},\uprho )$ is a Drinfeld $R$-module over a field $L$, the set $\{ \uprho_{\upalpha}|\; \upalpha\in\mathfrak{a}_{R}\}$ is a principal left ideal of 
$L\{ \uptau\}$
and has a unique generator $\uprho_{\mathfrak{a}_{R}}$ of ${\rm sgn}$ 1.
Then there exists a unique Drinfeld $R$-module defined over $L$, denoted \[ \mathfrak{a}_{R}\ast\D=(\C_{\infty},\mathfrak{a}_{R}\ast \uprho),\] for which
\[  \uprho_{\mathfrak{a}_{R}} \circ \uprho_{a} =(\mathfrak{a}_{R}\ast \uprho)_{a} \circ \uprho_{\mathfrak{a}_{R}} , \]
for all $a\in R$.
That is, $\uprho_{\mathfrak{a}_{R}}$ defines an isogeny  \[ \uprho_{\mathfrak{a}_{R}}:\D\longrightarrow \mathfrak{a}_{R}\ast\D.\] 
See \cite{Hayes0}, page 182.    If $\mathfrak{a}_{R} =(a)$ is principal, $(a)\ast\uprho$ is isomorphic to $\uprho$.  Thus, if we restrict to invertible ideals, there is an induced action of the
 group ${\sf Cl}_{R}$ on isomorphism classes of Drinfeld $R$-modules.  See \cite{Hayes0}, page 185.

When $\uprho$ is a Hayes module, $\mathfrak{a}_{R}\ast \uprho$ is also a Hayes module, so that
the $\ast$ action preserves the set of Hayes modules, and this time, the principal ideals with a monic generator act trivially.  Thus,
the set of Hayes $R$-modules is a principal homogeneous space for the $\ast$ action of ${\sf Cl}^{1}_{R}$.
We have
\begin{align}\label{HayesThm}  \mathfrak{a}_{R}\ast \D = \D^{\upsigma_{\mathfrak{a}_{R}}}  \end{align}
where $\upsigma_{\mathfrak{a}_{R}}$ is the automorphism associated to the narrow class of $\mathfrak{a}_{R}$ by Class Field Theory, and where for any automorphism $\upsigma$ of $\C_{\infty}$, $\D^{\upsigma}=(\C_{\infty}, \uprho^{\upsigma})$, with $ \uprho^{\upsigma}= \upsigma\circ \uprho$.  The statements in this paragraph do not appear as such in \cite{Hayes0} (but see Theorem 8.5 of \cite{Hayes0}), however, they may be verified using the same techniques
as in the case when $R$ is Dedekind, c.f.\  \cite{Thak}, \S 3 and particularly Theorem 3.3.4.


The narrow Hilbert class field has the following relation to any rank 1 Hayes $R$-module $\D= (\C_{\infty}, \uprho )$.  For any $a\in R$, let
\[ H^{\rm norm}_{R}\] be the {\bf {\em normalizing field}}: the field generated over $K$ by the coefficients of $\uprho_{a}$.

\begin{theo}\label{normalizingfield} $H^{\rm norm}_{R}$ is independent of the choice of Hayes module $\D$ and $a\in R$, and is equal to $H_{R}^{1}$.
\end{theo}

\begin{proof}  The exponential $e_{\uprho}$ satisfies the functional equation: for all $a\in R$, 
\[ e_{\uprho} (x) = \uprho_{a}\circ e_{\uprho} (a^{-1}x) .\]
Then, as in \cite{Thak}, Remark 2.4.3 (1), the coefficients of $\uprho_{a}$ are determined recursively by an algebraic equation with parameters in the coefficients of $\uprho_{b}$, for any other choice of $b\in R$,
and vice verca.  Thus  $H^{\rm norm}_{R}$ is independent of the choice of $a\in R$.
The class group ${\sf Cl}^{1}_{R}$ acts transitively on the set of Hayes modules over $R$
via the $\ast$-action.  Moreover, the fundamental relation  $\uprho_{\mathfrak{a}_{R}} \circ \uprho = \uprho'\circ\uprho_{\mathfrak{a}_{R}}$
where $\uprho'=\mathfrak{a}_{R}\ast \uprho$, implies that the coefficient fields coincide, so $H^{\rm norm}_{R}$ is independent of the choice of Hayes module.  The rest of the proof is formally the same as that found in Proposition 3.3.1 of \cite{Thak}, page 70, where it is shown
that 
$ {\rm Gal}(H^{\rm norm}_{R}/K )$ is canonically isomorphic to ${\sf Cl}^{1}_{R}$ via an isomorphism compatible with that of class field theory.  It follows that $ H^{\rm norm}_{R}$ is $H_{R}^{1}$.
\end{proof}


\begin{coro} Let $\D = ( \C_{\infty}, \uprho )$ be a Hayes module.  Then the coefficients of the exponential $e_{\uprho}$ are in $H_{A}^{1}$.
\end{coro} 

\begin{proof} By Theorem \ref{normalizingfield}, the coefficients of $\uprho_{a}$ are in $H_{A}^{1}$.  Since  $e_{\uprho}$ conjugates multiplication by $a$ with $\uprho_{a}$ (see equation (\ref{fundform})), the coefficients
 of $e_{\uprho}$ are also in $H_{A}^{1}$. \end{proof}

We will need the following result concerning reduction of torsion.   Fix $\gotp_{R}\subset R$ a prime ideal. 
Let $L$ be a finite Galois extension of $K$, $\mathcal{O}_{L}\supset R$ the integral closure of $R$ in $L$: we note that $\mathcal{O}_{L}$ is Dedekind, by the Krull-Akizuki Theorem (see \cite{Neu}, page 77).  Let
$\gotP\subset \mathcal{O}_{L}$ be a prime ideal which divides $\gotp_{R}$ and denote by $\mathcal{O}_{(\mathfrak{P})}$ the localization of $\mathcal{O}_{L}$ at $\mathfrak{P}$.

For $\D = (\C_{\infty},\uprho )$ a rank 1 Drinfeld $R$-module defined over $L$, we recall that $\D$ is said to have good reduction with respect to $\mathfrak{P}$ if $\D$ is, up to isomorphism, a Drinfeld module 
with coefficients in the local ring $\mathcal{O}_{(\mathfrak{P})}\subset L$ and the reduction mod $\mathfrak{P}$ is also a rank 1 Drinfeld module.  In this case, we denote by $\widetilde{\mathbb{D}}=(\C_{\infty},\widetilde{\uprho})$
the reduction of $\D$ mod $\mathfrak{P}$.

Now we assume that $\D$ is a Hayes $R$-module, and so is defined over $L=H^{1}_{R}$.  Note then that $\D$ has coefficients in $\mathcal{O}_{L}$ (see Corollary 7.4 of \cite{Hayes0}) and good reduction
 at all primes because it is of rank 1.  
Let $\gotm_{R}\subset R$ be a fixed modulus (an integral ideal, {\it not} necessarily invertible) and recall that the $\gotm_{R}$-torsion module of $\D$ is defined
\[\D[\gotm_{R}]:=\left\{x\in \C_{\infty} \; : \;\;\textsl{ }\uprho_{\gotm_{R}}(x)=0\right\} .\]  

\begin{note}\label{TorsionNote} If $\D$ is analytically uniformized by $\C_{\infty}/\mathfrak{a}_{R}$ for $\mathfrak{a}_{R}$ an invertible ideal, then the exponential map identifies 
\[ \D[\gotm_{R}] \cong (\mathfrak{m}_{R}\mathfrak{a}_{R}^{-1})^{\ast} / \mathfrak{a}_{R} \]
Here we recall that for any fractional $R$-ideal $\mathfrak{n}_{R}$,
\[  \mathfrak{n}^{\ast}_{R} = \{ \upbeta\in K\; :\;\; \upbeta\mathfrak{n}_{R}\subset R\}, \]
and that $\mathfrak{n}_{R}^{-1}= \mathfrak{n}_{R}^{\ast}$ if $\mathfrak{n}_{R}$ is invertible.
\end{note}

\begin{prop}\label{Rank1Prop} Let $\mathfrak{m}_{R}$ be prime to $\mathfrak{c}$.  Then $\D[\gotm_{R}]$ is a free $R/\gotm_{R}$-module of rank {\rm 1}.
\end{prop}

\begin{proof}  The proof follows the structure of that of Proposition 1.4 on page 102 of \cite{Sil}.  We may assume without loss of generality that a uniformizing lattice $\Uplambda$ for $\D$ has been chosen
to be a fractional $R$-ideal.  Since $\mathfrak{m}_{R}$ is invertible, we have $\D[\gotm_{R}]\cong \mathfrak{m}^{-1}_{R}\Uplambda/\Uplambda$.  Using the Chinese remainder theorem, we may
write both $\D[\gotm_{R}]$ and $R/\mathfrak{m}_{R}$ as direct sums of primary factors, so that the Proposition splits along each factor.  That is, if $\mathfrak{p}_{R}\subset R$ is a prime not dividing
the conductor, 
 we are reduced to showing that the quotient
\[\mathfrak{b}_{R}/\mathfrak{p}^{e}_{R}   \mathfrak{b}_{R} :=  \mathfrak{m}^{-1}_{R}\Uplambda/ \mathfrak{p}^{e}_{R}    \mathfrak{m}^{-1}_{R}\Uplambda\]
is a rank 1 module over the local ring $R/\mathfrak{p}^{e}$, where $e$ is the exponent of $\mathfrak{p}_{R}$ in the prime factorization of $\mathfrak{m}_{R}$ (which exists, by Corollary \ref{PrimeToConductor} of section \S \ref{s:A}).
Write
\[ R' =R/\mathfrak{p}^{e}, \quad  \mathfrak{p}'_{R} = \mathfrak{p}_{R}/\mathfrak{p}^{e}_{R},\quad \mathfrak{b}_{R}' = \mathfrak{b}_{R}/ \mathfrak{p}^{e}_{R}\mathfrak{b}_{R},\]
and consider the quotient
\[  \mathfrak{b}_{R}' / \mathfrak{p}_{R}'\mathfrak{b}'_{R}   \cong \mathfrak{b}_{R}/\mathfrak{p}_{R}\mathfrak{b}_{R}
\]
as a vector space over $R/\mathfrak{p}_{R}$.  Since $\mathfrak{b}_{R}$ is a product of invertible ideals, it is invertible, and thus, by Theorem \ref{2gentheo}, it has the 2-generator property.  It follows then
that any two elements of 
$ \mathfrak{b}_{R}$ are $R$ dependent, so the dimension of $ \mathfrak{b}_{R}/\mathfrak{p}_{R}\mathfrak{b}_{R}$ over $R/\mathfrak{p}_{R}$ is $\leq 1$, however, if the dimension
were $0$, we would have $ \mathfrak{b}_{R}=\mathfrak{p}_{R} \mathfrak{b}_{R}$, which cannot be the case.  Thus the dimension is 1.  By Nakayama's Lemma applied
to $R'$ and the $R'$ module $\mathfrak{b}'_{R}$,  $\mathfrak{b}'_{R}$ is free of rank 1.
\end{proof}

\begin{note}\label{Rank1Note} If $\mathfrak{m}_{R} = mR$ is principal, it is invertible and $\D[\mathfrak{m}_{R}] \cong m^{-1}\Uplambda/\Uplambda$.  In this case, there is a canonical isomorphism
\[m^{-1}\Uplambda/\Uplambda \longrightarrow \Uplambda/m\Uplambda  \longrightarrow R/mR  ,\quad x + \Uplambda \longmapsto mx + m\Uplambda \longmapsto mx + mR .\]
Thus in the case, the conclusion in Proposition \ref{Rank1Prop} follows for $\mathfrak{m}_{R}$ principal, even if it is not prime to the conductor.
\end{note}

If $\D[\gotm_{R}]\subset L$,  then since $\D[\mathfrak{m}_{R}]$ consists of the roots of the monic polynomial $\uprho_{\mathfrak{m}_{R}}$, $\D[\mathfrak{m}_{R}]\subset\mathcal{O}_{L}\subset\mathcal{O}_{(\mathfrak{P})}$: it then
makes sense to reduce $ \D[\gotm_{R}]$ modulo $\mathfrak{P}$.  Denote by
  \[  \widetilde{\D}[\gotm_{R}]\]
  the $\gotm_{R}$-torsion points of the reduced Drinfeld module $\widetilde{\D}$.

\begin{lemm}\label{injtor}  Let $\mathfrak{m}_{R}\subset R$ be a (not necessarily invertible) modulus.  Suppose that $\D [\mathfrak{m}_{R}]\subset L$ and $\mathfrak{P}\nmid\mathfrak{m}_{R}$.  Then the reduction map \[\D[\gotm_{R}]\longrightarrow \widetilde{\D}[\gotm_{R}]\] is injective. 
\end{lemm}

\begin{proof}   We note that the Lemma is true for $\mathfrak{m}_{R}=(m)$ principal: indeed, the polynomial $\uprho_{m}$ defining the $m$-torsion points
is $\mathfrak{P}$ primitive (i.e.\ $\not\equiv 0$ mod $\mathfrak{P}$) and since we have good reduction at $\mathfrak{P}$, $\widetilde{\D}$ is a rank 1 Drinfeld module hence $\widetilde{\uprho}_{m}$ reduces to a separable polynomial.  Then, we may apply Hensel's Lemma to conclude that the reduction map on $m$-torsion is injective (in fact bijective).   For general $\mathfrak{m}_{R}$, not necessarily invertible,
we have that $\mathfrak{m}_{R}$ is finitely generated.  We first observe that we may choose generators $m_{1},\dots ,m_{k}$ of $\mathfrak{m}_{R}$ none of which are in $\mathfrak{P}$.  Indeed, by hypothesis, at least
one of them, say $m_{1}$, is not in $ \mathfrak{P}$. If $m_{2}\in \mathfrak{P}$ then we replace $m_{2}$ by $m_{1}+m_{2}\notin \gotP$.  We then proceed inductively, as needed, to replace any
generators in $\mathfrak{P}$ by generators that are not in $\mathfrak{P}$.  We claim that $\D[\gotm_{R}]=\D[m_{1}]\cap \cdots \cap \D[m_{k}]$. Indeed, the inclusion $\D[\gotm_{R}]\subset\D[m_{1}]\cap \cdots \cap \D[m_{k}]$ follows from
\[\D[\gotm_{R}]=\bigcap_{m\in \gotm_{R}}\D[m]. \] On the other hand, for $m=a_{1}m_{1}+\cdots +a_{k}m_{k}$, $\uprho_{m}=\uprho_{a_{1}}\uprho_{m_{1}}+\cdots +\uprho_{a_{k}}\uprho_{m_{k}}$ and the vanishing of all of the $\uprho_{m_{1}},\dots ,\uprho_{m_{k}}$ implies that of $\uprho_{m}$, which proves the claim. Since the reduction map commutes with the intersection this shows injectivity.  

\end{proof}


\begin{note}\label{RedFrob} Let $\D$ be a Hayes $R$-module, $\mathfrak{p}_{R}\subset R$ an invertible prime ideal and consider the isogeny
\[ \uprho_{\mathfrak{p}_{R}}: \D\longrightarrow \mathfrak{p}_{R}\ast \D\]
between Hayes modules.  Let $\mathfrak{P}$ in $L$ be above $\mathfrak{p}_{R}$.  Then the associated reduction map 
\[ \widetilde{\uprho}_{\mathfrak{p}_{R}}: \widetilde{\D}\longrightarrow \widetilde{\mathfrak{p}_{R}\ast \D} \]
is equal to $\uptau^{\deg(\mathfrak{p}_{R})}$, where 
$ \deg (\mathfrak{p}_{R}) =
 \dim _{\F_{q}} ( A/\mathfrak{p}_{R}A  )  $.
 See\footnote{Note however that Hayes uses the $p$-Frobenius $\upvarphi(x)=x^{p}$ in the statement of Corollary 3.8, 
 whereas we use the $q$-Frobenius $\uptau(x) = x^{q}=\upvarphi^{n}(x)$, which is why his statement differs slightly from ours.} Corollary 3.8 of \cite{Hayes0}.
\end{note}







\section{Shimura's Main Theorem for Drinfeld Modules over Rank 1 Orders}\label{s:D}

In this section we prove Shimura's Main Theorem for Complex Multiplication in the setting of the order $R$.  This was proved in \cite{DGV} in the case $R=A$.
We begin with a few remarks on Class Field Theory adapted to the order $R$.   

Let $\I_{K}$ be the group of $K$-id\`{e}les, $C_{K}= \I_{K}/K^{\times}$ the id\`{e}le class group. 
If $S$ is a finite set of places, $\I^{S}$ is the subgroup of $S$-id\`{e}les:
\[ \I_{K}^{S} =  \prod_{\mathfrak{p}\in S} K^{\times}_{\mathfrak{p}} \cdot \prod_{\mathfrak{p}\not\in S} A^{\times}_{\mathfrak{p}}.\]
When $S= \{ \infty\}$, we write $\I_{A}:=\I^{\{ \infty\}}_{K}$.  The subgroup 
\[  C_{A} =  \left( \I_{A}\cdot K^{\times}\right) /K^{\times} < C_{K} \]
corresponds to the Hilbert class field $H_{A}$ via reciprocity.

In this spirit, for an order $R\subset A$, we define \[ \I_{R} := K^{\times}_{\infty} \cdot \prod_{\mathfrak{p}\not=\infty} R^{\times}_{\mathfrak{p}} < \I_{A}.\]
 Then
the Hilbert class field $H_{R}$ of $R$ corresponds to the subgroup $C_{R}<C_{K}$ defined
\[ C_{R} = \left( \I_{R}\cdot K^{\times}\right) /K^{\times}.\]
The narrow Hilbert class field $H^{1}_{R}$ corresponds to 
\[ C^{1}_{R} = \left( \I^{1}_{R}\cdot K^{\times}\right) /K^{\times}\]
where 
\[ \I^{1}_{R} =(K_{\infty}^{\times})_{+}\times \prod_{\mathfrak{p}}  R^{\times}_{\mathfrak{p}}  \]
and where $(K_{\infty}^{\times})_{+} < K_{\infty}^{\times}$ is the subgroup of sign 1 elements. 

 If $\mathfrak{m}\subset A$ is an ideal with prime factorization $\mathfrak{m}=\prod\mathfrak{p}^{n_{\mathfrak{p}}}$, $\mathfrak{m}_{R}$ its contraction to $R$, the
narrow ray class field $K_{\mathfrak{m}_{R}}$ corresponds  to 
\[ C_{\mathfrak{m}_{R}} = \left(  \I_{\mathfrak{m}_{R}} \cdot K^{\times}\right)/ K^{\times} < C_{R}<C_{K}\] where
\[   \I_{\mathfrak{m}_{R}} =(K_{\infty}^{\times})_{+} \times \prod U_{\mathfrak{p},R}^{(n_{\mathfrak{p}})},\quad U_{\mathfrak{p},R}^{(n_{\mathfrak{p}})} := U_{\mathfrak{p}}^{(n_{\mathfrak{p}})} \cap R_{\mathfrak{p}} . 
\]
(We recall that $U_{\mathfrak{p}}^{(n)}$ is the $n$th higher unit group of $K^{\times}_{\mathfrak{p}}$.)

Given an extension $L/K$, let 
$\mathfrak{m}
$
be the product of primes of $A$ ramifying in $L/K$
and let $\mathfrak{m}_{R}$ be the contraction of $\mathfrak{m}$ to $R$.  Let 
\[  {\sf I}_{\mathfrak{m}_{R}}^{\mathfrak{c}} \]
be the group of $R$-ideals prime to both $\mathfrak{c}$ and  $\mathfrak{m}_{R}$; by Theorem \ref{bijectionprimetoc} of section \S \ref{s:A}, it is isomorphic to $ {\sf I}_{\mathfrak{m}}^{\mathfrak{c}}$, the group of ideals in $A$ prime to both $\mathfrak{c}$ and $\mathfrak{m}$.  
Thus we may consider the Artin symbol as defining a map
\[ (\cdot, L/K): {\sf I}_{\mathfrak{m}_{R}}^{\mathfrak{c}} \longrightarrow {\rm Gal}(L/K) .\]

 In the id\`{e}lic formulation, the reciprocity map 
\[  [\cdot, K]:\I_{K}\longrightarrow {\rm Gal}(K^{\rm ab}/K)\]
satisfies 
\[   [s,K]|_{L} = ((s), L/K) \]
where now $(s)\subset A$ is the $A$-ideal  associated to $s\in\I_{K}$. 

As we have seen in the previous sections, when working with the order $R$, it is convenient to work with ideal classes represented by an ideal prime to the conductor $\mathfrak{c}$;  similarly, it will often be useful to restrict to id\`{e}les prime to $\mathfrak{c}$.
Define
\begin{align}\label{DefIUpperR} \I_{R}< \I^{R} := K^{\times}_{\infty} \cdot \prod_{\mathfrak{p}| \mathfrak{c}} R^{\times}_{\mathfrak{p}} \cdot \prod'_{\infty\not=\mathfrak{p}\;\nmid \;\mathfrak{c}}  K^{\times}_{\mathfrak{p}} ,
\end{align}
where the primes range over those in $A$, and 
where the final product is restricted over the $R^{\times}_{\mathfrak{p}} =A^{\times}_{\mathfrak{p}}$ for $\mathfrak{p}\nmid \mathfrak{c}$.
\begin{prop}  Let  $\I^{R} $ be as above.  Then 
\begin{align}\label{RIdeles} \I_{K} = \I^{R}   \cdot K^{\times}.\end{align}
\end{prop}

\begin{proof} 
Given $\upalpha = (\upalpha_{\mathfrak{p}}) \in \I_{K}$, by the Approximation Theorem, for any $\upvarepsilon$, we may find $a\in K$ so that for all $\mathfrak{p}|\mathfrak{c}$, $|a-\upalpha_{\mathfrak{p}}^{-1}|<\upvarepsilon$.  In particular, given $N\in \N$, we may find $a\in K$ so that for all $\mathfrak{p}|\mathfrak{c}$,
\[ a\upalpha_{\mathfrak{p}} \in U^{(N)}_{\mathfrak{p}} = \text{$N$th higher unit group} := 1+\mathfrak{p}_{\mathfrak{p}}^{N}.\]
On the other hand, since $\mathfrak{c}\subset\mathfrak{p}$, the completion $\mathfrak{c}_{\mathfrak{p}}$ is a power of $\mathfrak{p}_{\mathfrak{p}}$,  
\[ \mathfrak{c}_{\mathfrak{p}} = \mathfrak{p}_{\mathfrak{p}}^{s_{\mathfrak{p}}},\quad s_{\mathfrak{p}} \in \N. \]
So choosing $N$ sufficiently large, $ a\upalpha_{\mathfrak{p}} \in 1+ \mathfrak{c}_{\mathfrak{p}}$, for all $\mathfrak{p}|\mathfrak{c}$.  But since $\mathfrak{c}\subset R$, $1+ \mathfrak{c}_{\mathfrak{p}}\in R^{\times}_{\mathfrak{p}}$.  Thus, $a\upalpha\in \I^{R}$, which proves (\ref{RIdeles}). \end{proof}

Using the Proposition, we may write
\begin{align}\label{altCK} C_{K} = \frac{ \I^{R}   \cdot K^{\times}}{ K^{\times}} \cong \frac{ \I^{R} }{  \I^{R}   \cap K^{\times}} .\end{align}
which is framed in terms of the particular order $A_{f}$.
By (\ref{RIdeles}) and the fact that $K^{\times} \subset {\rm Ker}( [\cdot, K])$,
\[  {\rm Image}\big( [\cdot, K] \big)  =  [\I_{K},K]  = [\I^{R},K].\]

Given $L/K$ an abelian extension, the integral closure of $R$ in $L$ is Dedekind.  Then it makes sense to speak of a prime  $\mathfrak{p}_{R}\subset R$ which is unramified in $L$.  Assume also $\mathfrak{p}_{R}$
is prime to the conductor, so that it is the contraction of a unique $\mathfrak{p}\subset A$.  Then $A_{\mathfrak{p}} = R_{\mathfrak{p}}$.   If
$\uppi_{\mathfrak{p}}$ is a uniformizer of $K^{\times}_{\mathfrak{p}}$, identified as an id\`{e}le by taking 1's at all other places, then if $\upsigma_{\mathfrak{p}}\in {\rm Gal}(L/K)$ denotes the Frobenius associated 
to $\mathfrak{p}$, then
\begin{align}\label{piandFrob} [\uppi_{\mathfrak{p}}, K] |_{L} = \upsigma_{\mathfrak{p}} .\end{align}

Let $\mathfrak{a}_{R}\subset {\sf I}^{\mathfrak{c}}_{R}$: by primeness to $\mathfrak{c}$, $\mathfrak{a}_{R}$ is the contraction of a unique ideal of $\mathfrak{a}\subset A$.   For any prime $\mathfrak{p}\subset A$ 
(not necessarily prime to $\mathfrak{c}$),
denote 
\[ \mathfrak{a}_{R, \mathfrak{p}} :=\mathfrak{a}_{R}R_{\mathfrak{p}}. \]  The quotient 
\[ K_{\mathfrak{p}}/ \mathfrak{a}_{R,\mathfrak{p}} \]
is then an $R$-module.

\begin{theo}\label{AfModIso}  Let $\mathfrak{a}_{R}\in {\sf I}^{\mathfrak{c}}_{R}$. Then there is an isomorphism of $R$-modules 
\[  K/\mathfrak{a}_{R}\cong   \bigoplus_{\mathfrak{p}\subset A} K_{\mathfrak{p}}/ \mathfrak{a}_{R,\mathfrak{p}}.\]
\end{theo}



\begin{proof}  We follow the general structure of the proof as presented in the case of Dedekind domains (see Lemma 8.1 of Chapter II, \S 8 of \cite{Sil}), making adjustments where necessary. We may assume without loss of generality that $\mathfrak{a}_{R}\subset R$: indeed, simply take $\upalpha\in \mathfrak{a}_{R}^{-1}$ and write $\mathfrak{a}'_{R}=\upalpha\mathfrak{a}_{R}\subset R$. 
If the isomorphism asserted in the Theorem exists for $\mathfrak{a}_{R}'$, we may obtain a corresponding isomorphism for $\mathfrak{a}_{R}$ by multiplying by $\upalpha^{-1}$.

Let us write $M=K/\mathfrak{a}_{R}$ and define $M[\mathfrak{p}^{\infty}]$ as the submodule of elements of $M$ killed by some power of $\mathfrak{p}_{R}=\mathfrak{p}\cap R$.  We claim
that the sum-of-coordinates map
\[ S:   \bigoplus_{\mathfrak{p}\subset A}  M[\mathfrak{p}^{\infty}]\longrightarrow M,\quad \upmu = (\upmu_{\mathfrak{p}})\longmapsto \sum  \upmu_{\mathfrak{p}} \]
is an isomorphism.   Let $\upmu\in {\rm Ker}(S)$.  Now for each $\mathfrak{p}$ let $e(\mathfrak{p})$ be the minimal exponent so that $\mathfrak{p}_{R}^{e(\mathfrak{p})}\upmu_{\mathfrak{p}} =(0)$.
Fix $\mathfrak{q}\subset A$ a prime and write 
\[ \mathfrak{d}_{R} := \prod_{\mathfrak{p}\not=\mathfrak{q}}\mathfrak{p}_{R}^{e(\mathfrak{p})} ,\]
so that $\mathfrak{d}_{R}\upmu_{\mathfrak{p}}=(0)$ for all $\mathfrak{p}\not=\mathfrak{q}$.  On the other hand, we have as well $\mathfrak{d}_{R}\upmu_{\mathfrak{q}}= \mathfrak{d}_{R}S (\upmu ) = (0)$.
We note that $\mathfrak{d}_{R}$ is prime to $\mathfrak{q}_{R}$, that is, $\mathfrak{d}_{R}+ \mathfrak{q}^{e(\mathfrak{q})}_{R} = R$.  Therefore
\[ \upmu_{\mathfrak{q}} R = \upmu_{\mathfrak{q}}  (\mathfrak{d}_{R}+ \mathfrak{q}_{R}^{e(\mathfrak{q})} )  = (0)\]
hence $\upmu_{\mathfrak{q}}=0$, and since $\mathfrak{q}$ is arbitrary, $\upmu=0$.  Thus $S$ is injective.  

On the other hand, let $m\in M$.  We note that $M$ is a torsion
$R$-module, so there exists $\upalpha \in R$, non constant (i.e.\ not an $A$-unit) such that $\upalpha m=0$.  The principal ideal $\upalpha A\subsetneq A$ has a nontrivial prime decomposition in $A$
\[ \upalpha A = \prod_{i=1}^{s} \mathfrak{p}_{i}^{e_{i}}\]
 which contracts to a {\it primary decomposition}\footnote{The contraction of the prime decomposition would also be a prime decomposition if the prime factors of $\upalpha A$ are prime to $\mathfrak{c}$.  Otherwise
it is not necessarily the case that $(\mathfrak{p}^{e})_{R} = (\mathfrak{p}_{R})^{e}$.  Nonetheless, the contraction of a power of a prime $\mathfrak{r}=(\mathfrak{p}^{e})_{R} $ is primary.}
\[ \upalpha R  = \mathfrak{r}_{1}\cap\cdots \cap  \mathfrak{r}_{s} .\]
If we write \[ \mathfrak{e}_{i }=  \bigcap_{j\not=i}\mathfrak{r}_{j} \supset \upalpha R\]
then the set of $\mathfrak{e}_{i}$ are coprime.  Indeed, there is no contraction $(\mathfrak{p}_{i})_{R}$ which divides all of the $\mathfrak{e}_{i}$.  If there were a prime in $R$ dividing all
of the $\mathfrak{e}_{i}$, this prime would be a contraction $\mathfrak{p}_{R}$ of $\mathfrak{p}\subset A$, by Proposition \ref{contractionprop}.  Then 
$ \mathfrak{p}_{R}\supset \upalpha R $ and expanding this inclusion to $A$ gives
\[ \mathfrak{p} \supset \mathfrak{p}_{R} A\supset \upalpha A,\]
which implies that $\mathfrak{p}|\upalpha A$, forcing $\mathfrak{p}=\mathfrak{p}_{i}$
for some $i$.  Thus, we conclude that
 $\sum \mathfrak{e}_{i} = R$. 
If we take $\upvarepsilon_{i}\in \mathfrak{e}_{i}$ with $\sum \upvarepsilon_{i}=1$
then $\upvarepsilon_{i}\equiv 1\mod \mathfrak{r}_{i}$ and $\upvarepsilon_{i}\equiv 0\mod \mathfrak{r}_{j}$.   Since 
\[ \mathfrak{r}_{i}\upvarepsilon_{i}\subset \mathfrak{r}_{i}\cap \mathfrak{e}_{i} =\upalpha R,\]
$\mathfrak{r}_{i}\upvarepsilon_{i}m=0$.  And since $(\mathfrak{p}_{i})_{R}^{e_{i}} \subset\mathfrak{r}_{i} =  (\mathfrak{p}^{e_{i}}_{i}) _{R}$, it follows that $\upvarepsilon_{i}m\in M[\mathfrak{p}_{i}^{\infty}]$.  Let $\upmu = (\upmu_{\mathfrak{p}})$ with $\upmu_{\mathfrak{p}_{i}} = \upvarepsilon_{i}m$
and $\upmu_{\mathfrak{p}}=0$ for all other primes.  Then $S( \upmu )=m$ and we conclude that $S$ is an isomorphism.

To finish the proof, we will show that the inclusion $K\hookrightarrow K_{\mathfrak{p}}$ induces an isomorphism
\[ T:M[\mathfrak{p}^{\infty}]\stackrel{\cong}{\longrightarrow} K_{\mathfrak{p}}/ \mathfrak{a}_{R,\mathfrak{p}} .\]  
In what follows, we omit parentheses around the power of a contraction and write simply
\[ \mathfrak{p}_{R}^{e}   := (\mathfrak{p}_{R})^{e}.\]
Let $\overline{\upalpha}\in {\rm Ker}(T)$; if $\upalpha\in K$ is a representative, 
\begin{align}\label{repinclusion} \mathfrak{p}_{R}^{e}\upalpha\subset\mathfrak{a}_{R}\end{align}
for some $e\geq 0$, and being in the kernel, we have $\upalpha\in \mathfrak{a}_{R,\mathfrak{p}}$. 

\vspace{3mm}

\noindent {\it Claim }1.  $\upalpha\in R_{\mathfrak{q}}$ for every prime $\mathfrak{q}\subset A$.  As a consequence, $\upalpha\in R$.

\vspace{3mm}

\noindent {\it Proof of Claim} 1.  When $\mathfrak{q}=\mathfrak{p}$ this follows trivially from $\upalpha\in \mathfrak{a}_{R,\mathfrak{p}}\subset R_{\mathfrak{p}}$.  If $\mathfrak{q}\not=\mathfrak{p}$
then we have 
\[ \upalpha \left(  \mathfrak{p}_{R} \right)_{\mathfrak{q}}^{e}  = \upalpha \left(  \mathfrak{p}_{R}^{e} \right)_{\mathfrak{q}} \subset \mathfrak{a}_{R,\mathfrak{q}}\subset R_{\mathfrak{q}},\]
where 1) the equality above follows from the fact that the sum and product are continuous with respect to completion at $\mathfrak{q}$, and 2) the first inclusion follows from (\ref{repinclusion}). But since $\mathfrak{q}\not=\mathfrak{p}$, 
$\left(  \mathfrak{p}_{R} \right)_{\mathfrak{q}} = R_{\mathfrak{q}}$.  This proves the first part of the claim.  
The proof of the second statement follows from Corollary \ref{CompletionIntersectCoro} of section \S \ref{s:A}:
\begin{align}\label{intersectionwithKislocal} \upalpha\in R_{\mathfrak{q}} \cap K = R_{(\mathfrak{q}_{R} )} \quad \text{for all }\mathfrak{q},\end{align}
hence $\upalpha\in \bigcap R_{(\mathfrak{q}_{R} ) }=R$.   $\lozenge$

\vspace{5mm}

For any prime $\mathfrak{q}\subset A$ and $x\in R$, define \[ {\rm ord}_{\mathfrak{q}_{R}} (x) :=  \max\left\{  n\in \N \cup \{ 0\} \; : \;\;
x R  \subset \mathfrak{q}_{R}^{n}  \right\} .\]
For $\overline{\upalpha}\in {\rm Ker}(T)$,  by {\it Claim} 1, $\upalpha\in R$, hence ${\rm ord}_{\mathfrak{q}_{R}} (\upalpha) $ is well-defined.

\begin{note}\label{NoteOnOrders} We point out that if $\mathfrak{q}\nmid \mathfrak{c}$, then ${\rm ord}_{\mathfrak{q}_{R}} (\upalpha)$ agrees with ${\rm ord}_{\mathfrak{q}} (\upalpha)$,
 is available for any $\upalpha\in K$ and has values that range over $\Z$.  If $\mathfrak{q}|\mathfrak{c}$,  $\mathfrak{q}$ is not invertible and ${\rm ord}_{\mathfrak{q}_{R}} (\upalpha)$  cannot be extended to negative integers and is available only for $\upalpha\in R$.
\end{note}

  Thus we have
\[ \upalpha R \subset \bigcap \mathfrak{q}_{R}^{{\rm ord}_{\mathfrak{q}_{R}} (\upalpha)}.\]
Since $\mathfrak{a}_{R}$ is prime to the conductor $\mathfrak{c}$, it is invertible, and may be written as a product of invertible prime ideals $\mathfrak{q}_{1,R},\dots, \mathfrak{q}_{r,R}$,
each of which is a contraction of a corresponding prime in $A$:
\[  \mathfrak{a}_{R}=  \mathfrak{q}_{1,R}^{e_{1}} \cdots \mathfrak{q}_{r,R}^{E_{A}} = \mathfrak{q}_{1,R}^{e_{1}} \cap \cdots \cap \mathfrak{q}_{r,R}^{E_{A}} . \]
In this case, defining ${\rm ord}_{\mathfrak{q}_{R}} (\mathfrak{a}_{R})$ as we did for $\upalpha$, we have
\[   \mathfrak{a}_{R} =  \bigcap \mathfrak{q}_{R}^{{\rm ord}_{\mathfrak{q}_{R}} (\mathfrak{a}_{R})}.\]
By the choice of $\upalpha$ in the kernel, it follows that \[ {\rm ord}_{\mathfrak{q}_{R}} (\upalpha \mathfrak{p}^{e}_{R} )\geq {\rm ord}_{\mathfrak{q}_{R}}(\mathfrak{a}_{R}).\]  

\vspace{3mm}

\noindent {\it Claim} 2.  For $\mathfrak{q}_{R}\not=\mathfrak{p}_{R}$, ${\rm ord}_{\mathfrak{q}_{R}} (\upalpha \mathfrak{p}^{e}_{R} ) ={\rm ord}_{\mathfrak{q}_{R}} (\upalpha)$.   In particular,
\[ {\rm ord}_{\mathfrak{q}_{R}} (\upalpha ) \geq {\rm ord}_{\mathfrak{q}_{R}} (\mathfrak{a}_{R}) .\] 

\vspace{3mm}

\noindent {\it Proof of Claim} 2.  Suppose the contrary: writing ${\rm ord}_{\mathfrak{q}_{R}} (\upalpha)=n$, that ${\rm ord}_{\mathfrak{q}_{R}} (\upalpha \mathfrak{p}^{e}_{R} )=n+k$ for $k>0$.  Then 
\begin{align}\label{inclusionprimary} \upalpha \mathfrak{p}^{e}_{R} \subset \mathfrak{q}_{R}^{n+k}  .\end{align}
Now the ideal $ \mathfrak{q}_{R}^{n+k}$ is primary, and by the inclusion (\ref{inclusionprimary}), for all $x\in  \mathfrak{p}^{e}_{R}$, $\upalpha x\in \mathfrak{q}_{R}^{n+k}$ implies either $\upalpha\in \mathfrak{q}_{R}^{n+k}$
or $x\in {\rm Rad}( \mathfrak{q}_{R}^{n+k}) = \mathfrak{q}_{R}$, where ${\rm Rad}(\mathfrak{b})$ is the radical of an ideal $\mathfrak{b}$.   But $\upalpha\not\in \mathfrak{q}_{R}^{n+k}$ therefore $x\in\mathfrak{q}_{R}$ i.e.\ $ \mathfrak{p}^{e}_{R}\subset\mathfrak{q}_{R}$.  But
${\rm Rad}(\mathfrak{p}_{R}^{e}) = \mathfrak{p}_{R}\subset {\rm Rad}(\mathfrak{q}_{R}) = \mathfrak{q}_{R}$, which contradicts
$\mathfrak{q}_{R}\not=\mathfrak{p}_{R}$, since $R$ is of dimension 1.  $\lozenge$

\vspace{5mm}

We now treat the case $\mathfrak{q}_{R}=\mathfrak{p}_{R}$.

\vspace{3mm}

\noindent {\it Claim} 3. 
$ {\rm ord}_{\mathfrak{p}_{R}} (\upalpha ) \geq {\rm ord}_{\mathfrak{p}_{R}} (\mathfrak{a}_{R}) .$

\vspace{3mm}

\noindent {\it Proof of Claim }3.   Let $\mathfrak{p}_{R,\mathfrak{p}}$ be the completion of $\mathfrak{p}_{R}$ in $K_{\mathfrak{p}}$.   Note that for all $n$, by continuity,
\begin{align}\label{completionprop} ( \mathfrak{p}_{R,\mathfrak{p}} )^{n} = ( \mathfrak{p}^{n}_{R})_{\mathfrak{p}} \quad \text{and} \quad ( \mathfrak{p}^{n}_{R})_{\mathfrak{p}} \cap R= \mathfrak{p}^{n}_{R}. \end{align}
 By hypothesis, $\upalpha \in \mathfrak{a}_{R,\mathfrak{p}}$.  Therefore 
\[   {\rm ord}_{\mathfrak{p}_{R,\mathfrak{p}} } (\upalpha )\geq   {\rm ord}_{\mathfrak{p}_{R,\mathfrak{p}} } (\mathfrak{a}_{R,\mathfrak{p}}) .\]
By (\ref{completionprop}), 
\[  {\rm ord}_{\mathfrak{p}_{R} } (\upalpha )  =   {\rm ord}_{\mathfrak{p}_{R,\mathfrak{p}} } (\upalpha ), \quad  {\rm ord}_{\mathfrak{p}_{R} } (\mathfrak{a}_{R}) =  {\rm ord}_{\mathfrak{p}_{R,\mathfrak{p}} } (\mathfrak{a}_{R,\mathfrak{p}}) ,\]
which proves the claim. $\lozenge$

\vspace{5mm}

By {\it Claims} 2.\ and 3.,
\[ \upalpha R \subset  \bigcap \mathfrak{q}_{R}^{{\rm ord}_{\mathfrak{q}_{R}} (\mathfrak{a}_{R})}= \mathfrak{a}_{R} ,\]
therefore  $\upalpha\in \mathfrak{a}_{R}$.  Thus $\overline{\upalpha}=0$ in $M[\mathfrak{p}^{\infty}]$. 
This shows that
 $T$ is injective.   

What remains is surjectivity.  Let $\overline{\upbeta}\in K_{\mathfrak{p}}/ \mathfrak{a}_{R,\mathfrak{p}} $
be represented by $\upbeta\in K_{\mathfrak{p}}$.   Suppose  first that $\mathfrak{p}\nmid \mathfrak{c}$.  By Proposition \ref{WhenLocalEqual} of section \S \ref{s:A}, the localizations $A_{(\mathfrak{p})}$ and $R_{(\mathfrak{p}_{R})}$ are  equal, thus
\[  A_{ \mathfrak{p}} = R_{ \mathfrak{p}} \quad \text{and}\quad \mathfrak{a}_{R,\mathfrak{p}} =  \mathfrak{a}_{ \mathfrak{p}} . \]
We now apply weak approximation to the Dedekind domain $A$, in which modular constraints are enforced with respect to the ideals dividing either $\mathfrak{a}$ or $\mathfrak{c}$:  that is,
we deduce the existence of $\upalpha\in K$ with
the following properties: 

 \vspace{3mm}
 
\begin{enumerate}
\item[\ding{202}] $ {\rm ord}_{\mathfrak{p}} (\upalpha - \upbeta)
 \geq {\rm ord}_{\mathfrak{p}} (\mathfrak{a}) $.  
 
 \vspace{2mm}
 
 \noindent We emphasize that here,  $ {\rm ord}_{\mathfrak{p}} $ is the order calculated with respect to the $A$-ideal $\mathfrak{p}$. However, since $\mathfrak{p}\nmid \mathfrak{c}$, this inequality implies a corresponding inequality for $ {\rm ord}_{\mathfrak{p}_{R}} $:
 \[ {\rm ord}_{\mathfrak{p}_{R}} (\upalpha - \upbeta) \geq {\rm ord}_{\mathfrak{p}_{R}} (\mathfrak{a}_{R}).\]
 Here, we note that since $\upalpha - \upbeta\in K_{\mathfrak{p}}$, ${\rm ord}_{\mathfrak{p}_{R}} (\upalpha - \upbeta)$ is defined as the order with respect to the completion of $\mathfrak{p}_{R}$, which is equal to the completion of $\mathfrak{p}$.  That is, $ {\rm ord}_{\mathfrak{p}_{R}} (\upalpha - \upbeta)=
  {\rm ord}_{\mathfrak{p}} (\upalpha - \upbeta)$ and ${\rm ord}_{\mathfrak{p}_{R}} (\mathfrak{a}_{R}) ={\rm ord}_{\mathfrak{p}} (\mathfrak{a})$.
  
   \vspace{3mm}

 \item[\ding{203}] 
 ${\rm ord}_{\mathfrak{q}} (\upalpha ) \geq {\rm ord}_{\mathfrak{q}} (\mathfrak{a})$ if $\mathfrak{q}|\mathfrak{a}$ and $\mathfrak{q}\not=\mathfrak{p}$.   
 
 \vspace{2mm}

 \noindent Since, for such $\mathfrak{q}$, $\mathfrak{q}\nmid \mathfrak{c}$, this is again equivalent to
 \[{\rm ord}_{\mathfrak{q}_{R}} (\upalpha ) \geq {\rm ord}_{\mathfrak{q}_{R}} (\mathfrak{a}_{R}),\]
 where we are using the observation in {\it Note} \ref{NoteOnOrders} that  ${\rm ord}_{\mathfrak{q}_{R}} (\upalpha ) $ is defined for any $\upalpha\in K$.
 In particular, for such primes,  $\upalpha\in \mathfrak{a}_{R,\mathfrak{q}}$.
 
  \vspace{3mm}
 
  \item[\ding{204}]  ${\rm ord}_{\mathfrak{q}} (\upalpha ) \geq {\rm ord}_{\mathfrak{q}} (\mathfrak{c})$ for all other primes.  
  
   \vspace{2mm}

 \noindent  For such primes, this inequality implies
  \[  \upalpha \in \mathfrak{c}_{\mathfrak{q}} \subset R_{\mathfrak{q}}.\]
However, we claim that $R_{\mathfrak{q}}=\mathfrak{a}_{R,\mathfrak{q}}$.  Indeed, note that since $\mathfrak{a}_{R}$ is prime to $\mathfrak{c}$, the contraction $\mathfrak{a}\cap R$ is $\mathfrak{a}_{R}$, and the contraction $\mathfrak{a}_{\mathfrak{q}} \cap R_{\mathfrak{q}}$ contains {\it a priori} $\mathfrak{a}_{R,\mathfrak{q}}$ =  the $\mathfrak{q}$-completion of $\mathfrak{a}_{R}=\mathfrak{a}\cap R$.  As both $\mathfrak{a}_{\mathfrak{q}}$
 and $R_{\mathfrak{q}}$ are open in $K_{\mathfrak{q}}$, any point in the intersection may be approximated by a limit of points in the dense subset $\mathfrak{a}_{R}$, hence the contraction $ \mathfrak{a}_{\mathfrak{q}} \cap R_{\mathfrak{q}}$ is equal to $\mathfrak{a}_{R,q}$.  But since
 $\mathfrak{q}\nmid \mathfrak{a}$,  $\mathfrak{a}_{\mathfrak{q}} = A_{\mathfrak{q}}$, thus $\mathfrak{a}_{R,q}=\mathfrak{a}_{\mathfrak{q}} \cap R_{\mathfrak{q}}= A_{\mathfrak{q}}\cap R_{\mathfrak{q}}=R_{\mathfrak{q}} $.  As a result, in this case we have
 \[ \upalpha\in\mathfrak{a}_{R,\mathfrak{q}}.\]

 \end{enumerate}
The inequality in \ding{202} implies that $\upalpha-\upbeta\in\mathfrak{a}_{R,\mathfrak{p}}$, hence
$T(\overline{\upalpha}) = \overline{\upbeta}$. If $e$ is a non-negative integer satisfying 
\[ e\geq  {\rm ord}_{\mathfrak{p}} (\mathfrak{a})-{\rm ord}_{\mathfrak{p}} (\upalpha)\] then since $\mathfrak{p}\nmid\mathfrak{c}$, ${\rm ord}_{\mathfrak{p}_{R} }={\rm ord}_{\mathfrak{p}}$
takes products to sums and
\[ {\rm ord}_{\mathfrak{p}_{R} }(\mathfrak{p}^{e}_{R} ) +  {\rm ord}_{\mathfrak{p}_{R} }(\upalpha) ={\rm ord}_{\mathfrak{p}_{R} }(\mathfrak{p}^{e}_{R} \upalpha) \geq {\rm ord}_{\mathfrak{p}_{R}} (\mathfrak{a}_{R}) \] which implies that
\[  \mathfrak{p}^{e}_{R} \upalpha \subset  \mathfrak{a}_{R,\mathfrak{p}}.\]
 For the primes occurring in \ding{203} the same is true:
\[ {\rm ord}_{\mathfrak{q}_{R}} (\mathfrak{p}^{e}_{R} \upalpha) \geq   {\rm ord}_{\mathfrak{q}_{R}}(\upalpha )\geq  {\rm ord}_{\mathfrak{q}_{R}} (\mathfrak{a}_{R})\] and 
\[  \mathfrak{p}^{e}_{R} \upalpha \subset  \mathfrak{a}_{R,\mathfrak{q}}.\]
 Finally, for the primes in \ding{204}, we have trivially 
 \[  \mathfrak{p}^{e}_{R} \upalpha \subset  \mathfrak{p}^{e}_{R} \mathfrak{c}_{\mathfrak{q}}\subset \mathfrak{c}_{\mathfrak{q}}\subset R_{\mathfrak{q}}= \mathfrak{a}_{R,\mathfrak{q}} .\]
 It follows that
 \[  \mathfrak{p}^{e}_{R} \upalpha \subset \bigcap \left(   \mathfrak{a}_{R,\mathfrak{q}}\cap K  \right) =\mathfrak{a}_{R},\]
i.e., $\overline{\upalpha}\in M[\mathfrak{p}^{\infty}]$.  This proves surjectivity when $\mathfrak{p}\nmid \mathfrak{c}$.  

Lastly, we consider the case $\mathfrak{p}|\mathfrak{c}$.  In this case
we have, for some $n\geq 1$, that
\[ \mathfrak{c}_{\mathfrak{p}}=\mathfrak{p}_{\mathfrak{p}}^{n}\subset R_{\mathfrak{p}} .\]
We point out here that $\mathfrak{p}^{n}\not=\mathfrak{p}^{n}_{R}$: in the above, we are taking the completion of the power of the prime
ideal $\mathfrak{p}\subset A$
using the fact that $A_{\mathfrak{p}}$ is a discrete valuation ring, hence every ideal in $A_{\mathfrak{p}}$ is a power of $\mathfrak{p}_{\mathfrak{p}}$.
We then replace condition \ding{202}  by 
\[ {\rm ord}_{\mathfrak{p}} (\upalpha - \upbeta)
 \geq n = {\rm ord}_{\mathfrak{p}} (\mathfrak{c}). \]
 The latter implies that
 \[ \upalpha \equiv \upbeta \mod R_{\mathfrak{p}}.\]
 However, by an argument identical to that found in \ding{204} above, $ R_{\mathfrak{p}}= \mathfrak{a}_{R,\mathfrak{p}}$.  Thus
 $T(\overline{\upalpha})=\overline{\upbeta}$.
 The remaining modular constraints are imposed to guarantee that $\overline{\upalpha}$ is in $M[\mathfrak{p}^{\infty}]$.  
 Conditions \ding{203}  and \ding{204} are identical to those above.
Choose $e$ non-negative so that 
\[   \mathfrak{p}^{e}\upalpha \subset \mathfrak{p}^{m}_{\mathfrak{p}} = \mathfrak{c}_{\mathfrak{p}}  . \]
Since $ \mathfrak{p}_{R}^{e}\upalpha\subset \mathfrak{p}^{e}\upalpha $ and $\mathfrak{c}_{\mathfrak{p}}  \subset R_{\mathfrak{p}}=\mathfrak{a}_{R,\mathfrak{p}}$,
 \[ \mathfrak{p}_{R}^{e}\upalpha\subset \mathfrak{a}_{R,\mathfrak{p}}.\]
 Along with conditions  \ding{203}  and \ding{204} we deduce as before that $\mathfrak{p}^{e}_{R} \upalpha \subset \mathfrak{a}_{R}$ and thus $\overline{\upalpha}\in M[\mathfrak{p}^{\infty}]$.
Thus
  $T$ is surjective in this case as well. This completes the proof of the Theorem.
\end{proof}

Recall the ideal group $\I^{R}<\I_{K}$, defined in equation (\ref{DefIUpperR}) at the beginning of this section. For every $x=(x_{v})\in\mathbb{I}^{R}$ the $R$-ideal generated by $x$ is defined \[(x):=\prod_{\gotp\subset A,\; \gotp\nmid \mathfrak{c}}\gotp_{R}^{{\rm ord}_{\gotp}(x_{\gotp})} \in {\sf I}^{\mathfrak{c}}_{R}.\]
Moreover, for every $\gota_{R}\subset R$, an ideal prime to $\mathfrak{c}$, we denote \[x\gota_{R} :=(x)\gota_{R} .\]
Note that \[(x\gota_{R})_{\gotp}=(x)\gota_{R} R_{\gotp}=x_{\gotp}\gota_{R} R_{\gotp}=x_{\gotp}\gota_{R,\gotp}.\]
Theorem \ref{AfModIso} now gives the following isomorphism of $R$-modules:\[K/x\gota_{R} \cong \bigoplus_{\gotp}K_{\gotp}/x_{\gotp}\gota_{R,\gotp}.\]
We then define, given $x\in\mathbb{I}^{R}$, a homomorphism $x:K/\gota_{R}\rightarrow K/x\gota_{R} $  so that the following diagram commutes
\begin{diagram}
K/\gota_{R}&\rTo^{x}&K/x\gota_{R}   \\
\dTo^{\simeq}& &\dTo_{\simeq} \\  \bigoplus_{\gotp}K_{\gotp}/\gota_{R,\gotp}&\rTo^{x}&\bigoplus_{\gotp}K_{\gotp}/x_{\gotp}\gota_{R,\gotp}  \\
\end{diagram}
where the bottom horizontal arrow is defined
\[(t_{\gotp})\mapsto (x_{\gotp}t_{\gotp}).\]

Recall that for an ideal $\mathfrak{n}_{R}\subset R$,  
\[ \mathfrak{n}^{\ast}_{R} := \{ \upbeta \in K\; :\;\; \upbeta\mathfrak{n}_{R}\subset R\} .\]
Thus $\mathfrak{n}_{R}^{\ast}\mathfrak{n}_{R} \subset R$ and if $\mathfrak{n}_{R}$ is invertible, $\mathfrak{n}_{R}^{\ast}=\mathfrak{n}_{R}^{-1}$.

\begin{lemm}\label{starunionlem} Let $\mathfrak{a}_{R} \in {\sf I }_{R}^{\mathfrak{c}}$.  Then
\[ K =  \bigcup_{\mathfrak{m}\subset A}  (\mathfrak{m}_{R} \mathfrak{a}_{R}^{-1} )^{\ast}
\]
where the $\mathfrak{m}_{R}$ vary over all contractions $\mathfrak{m}_{R}=\mathfrak{m}\cap R$ of integral ideals $\mathfrak{m}\subset A$.
\end{lemm}

\begin{proof} It is enough to assume $\mathfrak{a}_{R}\subset R$; as usual, $\mathfrak{a}\subset A$ is the extension of $\mathfrak{a}_{R}$ to $A$. In this case given $\upalpha =x/y\in K$, $x,y\in R$ let
\[  \mathfrak{m} = y \mathfrak{c} \mathfrak{a}_{R},\]
which is an $A$-ideal due to the factor $\mathfrak{c}$.
Then $\mathfrak{m}\subset y\mathfrak{a}_{R} \subset R$ hence $\mathfrak{m}_{R}=\mathfrak{m}$. 
It follows that
$ \upalpha \in   (\mathfrak{m}_{R} \mathfrak{a}_{R}^{-1} )^{\ast}$, since
\[ \frac{x}{y}  \mathfrak{m}_{R} \mathfrak{a}_{R}^{-1}  \subset \frac{x}{y} y\mathfrak{a}_{R}\mathfrak{a}_{R}^{-1} \in R. \]
\end{proof}


In what follows, we denote by $K^{\rm ab}_{{\infty}}$ the maximal abelian extension completely split at $\infty$.

\begin{theo}[Main Theorem of ``Complex Multiplication'' over Orders]\label{Sh}
Let \[\mathbb{D}=(\C_{\infty},\uprho)\]be a rank $\text{\rm 1}$ Drinfeld $R$-module. Let $\upsigma\in {\rm Aut}(\C_{\infty}/K)$ be such that $\upsigma|_{K^{\rm ab}_{{\infty}}}$ is in
the image of the reciprocity map and
 let $s\in\mathbb{I}^{R}$ be such that\[[s,K]=\upsigma|_{K^{\rm ab}_{{\infty}}}.\]Let $\gota_{R}$ be an ideal of $R$ 
 prime to $\mathfrak{c}$. Given an analytic isomorphism of $R$-modules
\[\Upsilon:\C_{\infty}/\gota_{R}\longrightarrow \mathbb{D},\] 
there exists an analytic isomorphism of $R$-modules
\[\Upsilon':\C_{\infty}/s^{-1}\gota_{R}\longrightarrow \mathbb{D}^{\upsigma},\]
whose restriction to $K_{\infty}$ is unique, 
such that the following diagram commutes:\[\begin{diagram}K/\gota_{R}&\rTo^{s^{-1}}&K/s^{-1}\gota_{R}\\\dTo^{\Upsilon}& &\dTo_{\Upsilon'}\\\mathbb{D}&\rTo^{\upsigma}&\mathbb{D}^{\upsigma}\end{diagram}\]
\end{theo}

\begin{proof}  

We begin with a standard observation: if $(\mathbb{D}_{1},\Upsilon_{1})$ is another rank 1 Drinfeld $R$-module isomorphic to $\mathbb{D}$, with $\Upsilon_{1}:\C_{\infty}/\mathfrak{a}_{1,R}\rightarrow \D_{1}$ an analytic isomorphism, then if the Theorem holds for $(\mathbb{D}_{1},\Upsilon_{1})$, it holds also for $(\mathbb{D},\Upsilon)$. 
The proof of this is straight-forward, formally identical to that appearing, for example, on pages 160-161 of \cite{Sil}.
This allows us to assume that 
\begin{enumerate}
\item[1.] $\D$ is a Hayes module i.e.\ it is sign normalized and its coefficients belong to $H_{R}^{1}$ = the narrow Hilbert class field of $R$.  See section \S \ref{s:C}.
\item[2.] $\gota_{R}\subset R$.
\end{enumerate}

Fix $\mathfrak{m}_{R}\subset R$ an ideal which is the contraction of $\mathfrak{m}\subset A$.  The first step will be to show the existence of a commutative diagram 
\begin{align}\label{1step} \begin{diagram}(\gotm_{R}\gota^{-1}_{R})^{\ast}/\gota_{R}&\rTo^{s^{-1}}&(\gotm_{R}s \gota^{-1}_{R})^{\ast}/s^{-1}\gota_{R}\\\dTo^{\Upsilon}& 
&\dTo_{\Upsilon'_{\gotm}}\\\mathbb{D}&\rTo^{\quad\;\;\upsigma}&\mathbb{D}^{\upsigma}\end{diagram}
\end{align}
where $\Upsilon'_{\mathfrak{m}}$ is the restriction of an analytical isomorphism $\C_{\infty}/s^{-1}\mathfrak{a}_{R}\rightarrow \D^{\upsigma}$. 
Let $L\subset K^{\rm ab}_{\infty}$ be a finite Galois extension of $K$ containing
 \begin{itemize}
\item the narrow Hilbert class field $ H^{1}_{R}$: so that $\D$ is defined over $L$.
\item the torsion submodule
 $\D [\mathfrak{m}_{R}]$. 
 \item the narrow ray class field
$K_{\mathfrak{m}_{R}}$.
 \end{itemize}
Let $\mathcal{O}_{L}\supset R$ be the integral closure of $R$ in $L$: it is equal to the integral closure of $A$ in $L$, hence is Dedekind.  
Choose a prime $\mathfrak{P}\subset \mathcal{O}_{L}$ with the following properties: if $\mathfrak{p}=\mathfrak{P}\cap A$, then
   \begin{itemize}
 \item[\ding{192}] $L$ is unramified at $\mathfrak{p}$.
 \item[\ding{193}] $\upsigma|_{L} =[s,K]|_{L}=\upsigma_{\mathfrak{p}}=(\mathfrak{p}, L/K)$.
 \item[\ding{194}] $\mathfrak{p}\nmid  \mathfrak{c}, \mathfrak{m}$.
 \end{itemize}
 As $\D$ is a Hayes module, the leading coefficients of its additive polynomials belong to the constant field of $H^{1}_{R}$, hence it has good reduction at $\mathfrak{P}$.
 The existence of such a $\mathfrak{P}$ satisfying \ding{192} --  \ding{194}  is guaranteed by the Chebotarev Density Theorem.

Let $\uppi\in \mathbb{I}_{K}$ be such that it has a uniformizer at the $\gotp$ component and is $1$ everywhere else. Then by (\ref{piandFrob})
\[ [\uppi,K]=(\gotp, L/K),\] 
hence $[s\uppi^{-1},K]$ acts trivially on $K_{\gotm_{R}}\subset L$.  
Recall from the beginning of this section that $K_{\gotm_{R}}$ is indexed via Class Field Theory by the subgroup 
\[ C_{\mathfrak{m}_{R}}= (\I_{\mathfrak{m}_{R}} \cdot K^{\times}) /K^{\times} < C_{K}.\]
Thus the reciprocity map induces an isomorphism
\begin{align}\label{reciprocitymapforrcf} [\cdot, K]: \I_{K}/ \big( \I_{\mathfrak{m}_{R}}\cdot K^{\times}\big) \stackrel{\cong}{\longrightarrow} {\rm Gal}(K_{\mathfrak{m}_{R}}/K) . \end{align}
Therefore,
\[s\uppi^{-1}=\upalpha u\]
where $\upalpha\in K^{\times}$ and $u\in \I_{\mathfrak{m}_{R}}$.

By the hypothesis $\mathfrak{p}\nmid \mathfrak{c}$, we have $\mathfrak{p}_{R} \nmid \mathfrak{c}$, which implies $\mathfrak{p}_{R} $
defines a class of  the narrow class group ${\sf Cl}^{1}_{R}$.
Hence the isogeny
\[\uprho_{\gotp_{R}}: \D\longrightarrow \gotp_{R}\ast \D=\D^{\upsigma}\]
is defined,
where the equality on the right hand side comes from (\ref{HayesThm}). 
We recall by {\it Note} \ref{RedFrob} that its reduction mod $\mathfrak{P}$ satisfies $\widetilde{\uprho_{\gotp_{R}}}=\tau^{\deg(\gotp_{R})}$.  
The key observation in the proof of the Main Theorem is that we can replace the {\it discontinuous
automorphism} $\upsigma$ by the {\it analytic endomorphism} $\uprho_{\mathfrak{p}_{R}}$ if we restrict to $\D[\mathfrak{m}_{R}]$: that is, 
we claim that the following diagram commutes
\begin{equation}\label{diagram1}\begin{diagram}\D[\gotm_{R}]&\rTo^{\uprho_{\gotp_{R}}}&\D^{\upsigma}[\gotm_{R}]\\\dInto& &\dInto\\\mathbb{D}&\rTo^{\upsigma}&\mathbb{D}^{\upsigma}\end{diagram},\end{equation}
where the vertical arrows are the inclusions.
Indeed, for every torsion point $t\in \D[\gotm_{R}]$, we have the equalities mod $\mathfrak{P}$
\[\widetilde{\uprho_{\gotp_{R}}(t)}=\widetilde{\uprho_{\gotp_{R}}}(\, \widetilde{t}\,)= \widetilde{t}^{\;\widetilde{\upsigma}}= \widetilde{\;t^{\upsigma}}\] 
since $\widetilde{\uprho_{\gotp_{R}}}=\tau^{\deg(\gotp_{R})}=\widetilde{\upsigma}$. 
As $\mathfrak{P}\nmid\mathfrak{m}_{R}$, by Lemma \ref{injtor}, the reduction map 
\[ \D^{\upsigma}[\mathfrak{m}_{R}]\rightarrow\widetilde{ \D^{\upsigma}}[\mathfrak{m}_{R}]\]
is injective, hence
\[  \uprho_{\gotp_{R}}(t) = t^{\upsigma}. \]

By hypothesis \ding{194} above, $\mathfrak{p}\nmid \mathfrak{c}$, so $\mathfrak{p}_{R}$ is invertible. We now fix\[\Upsilon'':\C_{\infty}/\gotp_{R}^{-1}\gota_{R}\longrightarrow \D^{\upsigma}\]an analytic isomorphism such that the following diagram commutes:
\[\begin{diagram}\C_{\infty}/\gota_{R}&\rTo^{\sf can}&\C_{\infty}/\gotp_{R}^{-1}\gota_{R}\\\dTo^{\Upsilon}& &\dTo_{\Upsilon''}\\\mathbb{D}&\rTo_{\uprho_{\gotp_{R}}}&\mathbb{D}^{\upsigma}.\end{diagram}\]
where {\small ${\sf can}$} is the canonical map $x + \gota_{R}\mapsto x+\gotp_{R}^{-1}\gota_{R}$.
The existence of $\Upsilon''$ is a consequence of the fact that the kernel of the additive homomorphism $\uprho_{\mathfrak{p}_{R}}$ is $\D [\mathfrak{p}_{R}]$, whose pre-image 
by the analytical $R$-module isomorphism $\Upsilon$ is exactly $\mathfrak{p}_{R}^{-1} \mathfrak{a}_{R}$.

Using the decomposition $s=\upalpha\uppi u$, 
we have \[(s)=(\upalpha)(\uppi)=(\upalpha)\gotp_{R}.\]Therefore $s^{-1}\gota_{R}=(s^{-1})\gota_{R}=\upalpha^{-1}\gotp_{R}^{-1}\gota_{R}$, and 
 multiplication by $\upalpha^{-1}$ gives the following 
isomorphism:\[\upalpha^{-1}:\C_{\infty}/\gotp_{R}^{-1}\gota_{R}\longrightarrow \C_{\infty}/s^{-1}\gota_{R}.\]We can then
form the following 
diagram:\begin{equation}\label{diagram2}\begin{diagram}\C_{\infty}/\gota_{R}&\rTo^{\sf can}&\C_{\infty}/\gotp_{R}^{-1}\gota_{R}&\rTo^{\upalpha^{-1}}&\C_{\infty}/s^{-1}\gota_{R}\\\dTo^{\Upsilon}& &\dTo_{\Upsilon''}& 
&\dTo_{\Upsilon'}\\\mathbb{D}&\rTo_{\uprho_{\gotp_{R}}}&\mathbb{D}^{\upsigma}&\rTo_{\rm id}&\mathbb{D}^{\upsigma}\end{diagram}\end{equation}where $\Upsilon'$ is the unique analytic isomorphism making the diagram  commute. 
By {\it Note} \ref{TorsionNote},  $\D[\mathfrak{m}_{R}]$ may be identified with  
\[  (\mathfrak{m}_{R}\mathfrak{a}_{R}^{-1})^{\ast}/\mathfrak{a}_{R} ;\]
if $\mathfrak{m}_{R}$ is invertible, this reduces to the usual $\mathfrak{m}_{R}^{-1}\mathfrak{a}_{R}/\mathfrak{a}_{R}$.
\begin{clai}  For all $t\in(\mathfrak{m}_{R}\mathfrak{a}_{R}^{-1})^{\ast}/\mathfrak{a}_{R}$, 
\[\Upsilon(t)^{\upsigma}=\Upsilon'(s^{-1}t).\]
\end{clai}

\vspace{3mm}

\noindent {\it Proof of the Claim}:
By (\ref{diagram1}) we know that $\Upsilon(t)^{\upsigma}=\uprho_{\gotp_{R}}(\Upsilon(t))$ for every $t\in(\mathfrak{m}_{R}\mathfrak{a}_{R}^{-1})^{\ast}/\mathfrak{a}_{R}$. 
Using the commutativity of (\ref{diagram2}) the statement reduces to proving that \[\Upsilon'(\upalpha^{-1}t)=\Upsilon'(s^{-1}t).\]
This is equivalent to showing that
 \[\upalpha^{-1}t-s^{-1}t\in s^{-1}\gota_{R}\]for all $t\in (\mathfrak{m}_{R}\mathfrak{a}_{R}^{-1})^{\ast}/\mathfrak{a}_{R}$. Or equivalently, by Theorem \ref{AfModIso},
  \[\upalpha^{-1}t-s_{\gotq}^{-1}t\in s_{\gotq}^{-1}\gota_{R,\gotq}\]for all $t\in (\gotm_{R}\gota_{R,\gotq}^{-1})^{\ast}$ and all $\gotq$ a prime ideal in $A$. 
  Since $s_{\gotq}=\upalpha\uppi_{\gotq}u_{\gotq}$, we are reduced to showing that\[\uppi_{\gotq}u_{\gotq}t-t\in \gota_{R,\gotq}. \] 
  
First, note that since $u_{\gotq}\equiv 1 \mod  \mathfrak{m}_{R,\mathfrak{p}} $, $u_{\gotq} = 1+m$, $m\in  \mathfrak{m}_{R,\mathfrak{p}} $ and since 
$t\in(\gotm_{R}\gota_{R,\gotq}^{-1})^{\ast}$, 
$tm\in \mathfrak{a}_{R,\mathfrak{q}}$.  Thus
\[ \uppi_{\gotq}u_{\gotq}t-t\in \gota_{R,\gotq}\quad\Longleftrightarrow \quad \uppi_{\mathfrak{q}}t-t\in\mathfrak{a}_{R,\mathfrak{q}}.\]
Now if $\mathfrak{q}\not=\mathfrak{p}$, $\uppi_{\mathfrak{q}} =1$ and the result follows immediately since $\uppi_{\mathfrak{q}}t-t =0$.  If 
$\mathfrak{q}=\mathfrak{p}$, then since $\mathfrak{p}\nmid \mathfrak{c}, \mathfrak{m}$ we have 
\[   R_{\mathfrak{p}} = A_{ \mathfrak{p}} , \quad \gota_{R,\mathfrak{p}} = \mathfrak{a}_{\gotp}  \]
and 
\[  \mathfrak{m}_{R,\mathfrak{p}} = A_{\mathfrak{p}}.\]
It follows that $ (\mathfrak{m}_{R}^{\ast} )_{\mathfrak{p}} = \mathfrak{m}_{R}^{\ast} R_{\mathfrak{p}} = A_{\mathfrak{p}}$, hence $t\in  (\gotm_{R}\gota_{R,\gotp}^{-1})^{\ast}$
implies $t\in \gota_{\mathfrak{p}}$.  Since $\uppi_{\mathfrak{p}}$ is a uniformizer, $ \uppi_{\mathfrak{p}}-1$ is a unit and so
\[  \uppi_{\mathfrak{p}}t-t  = (\uppi_{\mathfrak{p}}-1)t \in \gota_{R,\mathfrak{p}}.  \]
This proves the claim.   $\lozenge$

\vspace{3mm}

The choice $\Upsilon'_{\mathfrak{m}} = \Upsilon'$ makes the diagram (\ref{1step}) commute.
By Lemma \ref{starunionlem}, $K/\gota_{R}=\bigcup_{\gotm}(\mathfrak{m}_{R}\mathfrak{a}_{R}^{-1})^{\ast}/\gota_{R}$, therefore it is enough to
show that these diagrams are compatible and so fit together to produce the diagram appearing in the statement of the Main Theorem.  
Let $\gotn\subset \gotm$; note that \[ \upxi:=\Upsilon'_{\gotn}\circ {\Upsilon'_{\gotm}}^{-1}\in {\rm Aut}(\mathbb{D}^{\upsigma})=R^{\times}= \mathbb{F}^{\times}_{q}  ;\] 
in the above, we are using the fact that $\D^{\upsigma}$ is a Drinfeld module with lattice identified with $\mathfrak{p}_{R}^{-1}\mathfrak{a}_{R}$, hence, by Theorem 1.6 of \cite{Hayes0}, the endormorphism ring
is equal to $R$ and thus $ {\rm Aut}(\mathbb{D}^{\upsigma})=R^{\times}$.

It will be enough to show that \[\Upsilon'_{\gotn}|_{(\gotm_{R} s\gota_{R}^{-1})^{\ast}/s^{-1}\gota_{R}}=\Upsilon'_{\gotm}.\] This follows, since \[\upxi \Upsilon'_{\gotm}(s^{-1}t)=\Upsilon'_{\gotn}(s^{-1}t)=\Upsilon(t)^{\upsigma}=\Upsilon'_{\gotm}(s^{-1}t), \] so that in particular, $\upxi$ acts as the identity on $\mathbb{D}^{\upsigma}[\gotm_{R}]$ and thus $\upxi=1$. 
\end{proof}

\section{Hayes Theory Over Rank 1 Orders}\label{RCFofR}

In this section, we will present what is commonly known as Hayes Theory for the order $R\subset A$.  That is, we will prove generation theorems for the Hilbert class field $H_{R}$ and
the ray class fields $K_{\mathfrak{m}_{R}}$, defined in the previous sections.  

As we have already indicated, this topic was initiated by Hayes in his orders paper \cite{Hayes0}, however
this section contains material not proved in Hayes' original work.  In the first place, we will generate $H_{R}$ using the modular invariant $j(\mathfrak{a}_{R})$ of any invertible ideal $\mathfrak{a}_{R}\subset R$,
in the style of \cite{DGV} and \cite{DGIII}: in contrast with  \cite{Hayes0}, \S 8, where $H_{R}$ is identified with the field of invariants of a Drinfeld $R$-module $\uprho$ with lattice homothetic to $\mathfrak{a}_{R}$.  Secondly,
in \cite{Hayes0}, Hayes restricts to $R=A$ in his Theorem on the generation of the ray class fields $K_{\mathfrak{m}_{R}}$ using torsion elements of $\uprho$: in this section, we prove this in this case
where $R$ is not necessarily equal to $A$.  

The main tool will be Theorem \ref{Sh}, proved in \S \ref{s:D}. In particular, the strategy will be to adapt the arguments found in \cite{DGV} to the order $R$,
discussing in detail only those results whose proofs require a non trivial modification in this case.

First some notation.   In this section we will drop the subscript $R$ from ideals in $R$, as we will not need to pass from ideals in $A$ to ideals in $R$.
Write
\[ K_{\infty}= \F_{\infty}((u_{\infty}))\]
where $u_{\infty}$ is a uniformizer at $\infty$ and $\F_{\infty}\supset\F_{q}$ is the field of constants.    We also choose a sign function
\[ {\rm sgn}:K_{\infty}^{\times}\longrightarrow \F^{\times}_{\infty}\]
so that if $x=c_{N}u_{\infty}^{N} + $ lower order terms, then $\text{sgn}(x)=c_{N}$. 
Let $S$ be a set of representatives of $\F^{\times}_{\infty}/\F^{\times}_{q}$, where we assume $1\in S$.
We say that $f\in K_{\infty}$ is positive if ${\rm sgn}(f)\in S$.   

Let $\mathfrak{a}\subset R$ be an ideal. Denote by 
\[ \mathfrak{a}^{+} = \{x\in\mathfrak{a}|\; {\rm sgn}(x)\in S\}; \]
define\[\zeta^{\gota}(n):=\sum_{x\in \gota^{+}}x^{-n}\]for any $n\in \mathbb{N}$. 
The definition of $\upzeta^{\mathfrak{a}}(n)$ is independent of the choice of signs $S$, see Note 3 of \cite{DGV}. 

Now assume $A\not=\F_{q}[T]$.  Writing \begin{align}\label{defnofJ} J(\gota):=\frac{\zeta^{\gota}(q^{2}-1)}{\zeta^{\gota}(q-1)^{q+1}}, \end{align} we define 
\begin{align}\label{defnofj} j(\gota):=\frac{1}{\frac{1}{T^{q}-T}-\frac{T^{q^{2}}-T}{(T^{q}-T)^{q+1}}J(\gota)}.\end{align}
If $\upalpha\in K^{\times}$ has sgn one then
$ ((\upalpha) \mathfrak{a})^{+}=\upalpha\mathfrak{a}^{+}$.
This implies that $ j((\upalpha)\mathfrak{a})=  j(\mathfrak{a})$ and we obtain
a well-defined function on the narrow class group
\[  j=j_{R}:{\sf Cl}^{1}_{R}\longrightarrow  \mathbb{C}_{\infty}.\]
In fact, it is a class invariant:

\begin{prop}\label{classinv} $j$ induces a well-defined function 
\[  j:{\sf Cl}(A)\longrightarrow  \mathbb{C}_{\infty}.\]
\end{prop}

\begin{proof} Formally the same as that of Proposition 1 of \cite{DGV} .\end{proof}


Since $R$ is Noetherian, it is a finitely generated $\F_{q}$-algebra:
\begin{align}\label{Abasis} R =\F_{q}[f_{1},f_{2},\dots ,f_{N}]   = \langle 1, f_{1},f_{2},\dots , f_{N}, f_{N+1},\dots \rangle_{\F_{q}}\end{align}
where the presentation on the far right is that of an $\F_{q}$ vector space and where the additional vector space generators $f_{N+1},\dots$ complement the ring generators $f_{1},\dots , f_{N}$.  We assume
that $0<\deg (f_{1})<\deg (f_{2})<\cdots $. 


Fix $\mathfrak{a}\subset R$ a non-principal invertible ideal: by Theorem \ref{2gentheo}, we may
write $\mathfrak{a}=(g,h)$; we suppose that $\deg (g)<\deg (h)$.  Without loss of generality, we may assume
\begin{itemize}
\item[1.] $g$ has the smallest degree of all positive non-zero
elements of $\mathfrak{a}$.  Indeed, by {\it Note} \ref{mingennote}, for any non-zero $x\in\mathfrak{a}$, there exists $y\in\mathfrak{a}$ with $\mathfrak{a}=(x,y)$.   
\item[2.] The degree of $h$ is minimal amongst $h'$ with $\mathfrak{a}=(g,h')$ and in particular, $\deg {h}\not= \deg (fg)$ for all $f\in R$ (otherwise, we may replace $h$ by $h'= ch + fg$, where $c\in\F_{q}$ is chosen so that $\deg (h') < \deg (h)$.)
\end{itemize}

\begin{theo}\label{jadifffrom1}  For all $\mathfrak{a}\subset R$ a non principal invertible ideal, $j(\mathfrak{a})\not=j((1))$. 
\end{theo}

\begin{proof}  The proof of this result is formally the same as that found in \S 6 of \cite{DGV}.
 \end{proof}
 
 The following is the analog for $R$ of a Theorem of Goss:
 
 \begin{theo}\label{GossTheorem} Let $\uprho$ be an  Hayes $R$-module with lattice $\upxi_{\uprho}\mathfrak{a}$, $\mathfrak{a}\subset R$ an invertible ideal. 
 Then for all $n\in\N$,
\[  \upxi_{\uprho}^{-n}  \sum_{0\not=x\in \mathfrak{a}}  x^{-n} \in H^{1}_{R}. \]
\end{theo}

\begin{proof}  Since for each $a\in R$, the coefficients of $\uprho_{a}$ are in $H^{1}_{R}$, the same is true of the corresponding exponential $e_{\uprho}$.
 The proof is then identical to that in the case of $R=A$ e.g.\ see Theorem 5.2.5 on page
159 of \cite{Thak}.
\end{proof}


We now use the Main Theorem to prove the following key result.    

\begin{theo}\label{t3}  
Let $\mathfrak{a}\in {\sf I}_{R}^{\mathfrak{c}}$ and let $s\in \I^{R}$ be a $K$-id\`{e}le prime to $\mathfrak{c}$.  Then $J(\mathfrak{a})\in K^{\rm ab}_{\infty}$
and 
\[  J(s^{-1}\mathfrak{a}) = J(\mathfrak{a})^{[s,K]}. \]
\end{theo}

\begin{proof}
We first show that $J(\gota)\in H^{1}_{R}$. 
Let us call
\begin{align}\label{normzeta}\widetilde{\upzeta}^{\gota}(n):=\frac{\upzeta^{\gota}(n)}{{\upxi_{\uprho}}^{n}}\end{align}
the \textsl{normalized} 
value of $\upzeta^{\gota}(n)$, where $\upxi_{\uprho}$ is the normalizing transcendental element of $\C_{\infty}$ that corresponds to $\D = (\C_{\infty}, \uprho)$, a Hayes module associated to the ideal $\mathfrak{a}$, see section \S \ref{s:C} of this paper.
As we shall see, the choice of $\D$ amongst the $(q^{d_{\infty}}-1)/(q-1)$ Hayes modules associated to $\mathfrak{a}$ will not effect our arguments.
Because $J(\gota)$ is a homogeneous ratio of values of the 
zeta function $\upzeta^{\gota}$ (see equation (\ref{defnofJ})), we may replace these values by their normalized values. 
Then, we claim, by Theorem \ref{GossTheorem}, that the normalized values are all in $H^{1}_{R}$. Indeed, we have $\upxi_{\uprho}^{-n}\cdot \sum_{0\not= x\in\mathfrak{a}} x^{-n}\in H^{1}_{R}$.
  However, since $\mathfrak{a}\setminus \{ 0\} = \bigsqcup_{c\in\F_{q}^{\times}} c\mathfrak{a}^{+}$,  
  \[  \sum_{0\not= x\in\mathfrak{a}} x^{-n}=\# \F_{q}^{\times}\cdot \upzeta^{\mathfrak{a}}(n)= - \upzeta^{\mathfrak{a}}(n) .\]
 Let $\upsigma=[s,K]$. 
We claim that for $(q-1)|n$, \begin{equation}\label{normzetasareconj} \widetilde{\upzeta}^{\gota}(n)^{\upsigma}=\widetilde{\upzeta}^{s^{-1}\gota}(n).\end{equation}
First note that if $e_{\uprho}(z)=e_{\upxi_{\uprho}\gota}(z)$ is the associated exponential function, then taking its logarithmic derivative we get
\[\frac{1}{e_{\uprho}(z)}=\sum_{\upalpha\in\mathfrak{a}} \frac{1}{z+\upxi_{\uprho}\upalpha}
=
-\sum_{n=0}^{\infty}\sum_{\upalpha\in \gota}\frac{z^{n}}{(\upxi_{\uprho}\upalpha)^{n+1}}
=- \sum_{n=0}^{\infty}\widetilde{\upzeta}^{\gota}(n+1)z^{n}.\]
Therefore 
\[ e_{\uprho}(z)=\sum_{n=0}^{\infty}c_{n}z^{q^{n}}, \quad c_{n}\in H^{1}_{R}, \]
 where the $c_{n}$ are algebraic 
combinations of the $\widetilde{\upzeta}^{\gota}(n+1)$ of a universal form which is dictated by the formula for the reciprocal of a power series. 
Fix $a\in R$ and write \[\uprho_{a}(\tau)=a+g_{1}\tau+\cdots +g_{d}\tau^{d},\quad g_{1},\dots ,g_{d}\in H_{R}^{1} .\]
Then, the equation
\[e_{\uprho}(az)=\uprho_{a}(e_{\uprho}(z))\]
implies 
that 
\begin{align*} az+c_{1}a^{q}z^{q}+(c_{2}a^{q^{2}})z^{q^{2}}+\cdots & = \uprho_{a}(z+c_{1}z^{q}+\cdots ) \\
& =\uprho_{a}(z)+\uprho_{a}(c_{1}z^{q})+\cdots \\ &=az+(g_{1}+ac_{1})z^{q}+(g_{2}+c_{1}^{q}g_{1}+ac_{2})z^{q^{2}}+\cdots \end{align*}
The shape of the last expression above is again of a universal nature and depends only on the coefficients of $\uprho_{a}$ and the coefficients $e_{\uprho}(z)$.
That is, we have \[c_{1}=\frac{g_{1}}{a^{q}-a}, \quad
c_{2}=\frac{g_{2}+c_{1}^{q}g_{1}}{a^{q^{2}}-a},\quad \dots \]
and the coefficients of $e_{\uprho}(z)$ may be solved for in terms of the coefficients of $\uprho_{a}$ using a universal recursion.  In particular, we have given a formula
for the coefficients of the normalized exponential attached to any Hayes module $(\D,\uprho )$, which only depends on the coefficients of $\uprho_{a}$ for $a\in R$ fixed.
As $\mathbb{D}^{\upsigma}=(\C_{\infty},\uprho^{\upsigma})$ has lattice homothetic to $s^{-1}\gota$ by Theorem \ref{Sh}, 
it follows that
\[e_{\uprho^{\upsigma}}(z)=   e_{\upxi_{\uprho^{\upsigma}} s^{-1}\mathfrak{a}} (z)   =\sum_{n=0}^{\infty}c_{n}^{\upsigma}z^{n}  \]
where $\upxi_{\uprho^{\upsigma}} $ is the transcendental factor associated to $\D^{\upsigma}$.
Therefore for all $n$
\begin{align}  \frac{\upzeta^{s^{-1}\mathfrak{a}}(n)  }{\upxi^{n}_{\uprho^{\upsigma}} } =  \widetilde{\upzeta}^{s^{-1}\gota}(n) = (\widetilde{\upzeta}^{\gota}(n))^{\upsigma}.\end{align}
The statement about $J$ follows immediately.
\end{proof}

\begin{theo}\label{jt3} For all $\mathfrak{a}\in {\sf Cl}_{R}$, $j(\mathfrak{a})\not=\infty$.  In particular, for any $s\in \I^{R}$, \[  j(s^{-1}\mathfrak{a}) = j(\mathfrak{a})^{[s,K]}. \]
\end{theo}

\begin{proof}  The proof is identical to that of Theorem 3, \cite{DGV}.
\end{proof}

\begin{coro}\label{cor}   The $j$-invariant  takes values in $H_{R}\subset H_{R}^{1}$, the function \[ j:{\sf Cl}_{R}\longrightarrow H_{R}\]
is injective and 
\[ H_{R}=K(j(\gota)).\]
\end{coro}
\begin{proof}  Identical to the proofs of Corollary 1 and Theorem 4 of \cite{DGV}.  \end{proof}




In what follows, we will need to pass back and forth between ideals of $A$ and their contractions in $R$, so we return to our notation which makes the distinction.

Let $\mathfrak{m}_{R}\subset R$ be an ideal contracted from a non-trivial ideal $\mathfrak{m}\subset A$, and $K_{\mathfrak{m}_{R}}$ the associated narrow ray class field. 
By Class Field Theory, $K_{\mathfrak{m}_{R}}\subset K^{\rm ab}_{\infty}$ is the fixed field of the group of Artin symbols $[s,K]\in {\rm Aut}(K^{\rm ab}_{\infty}/K)$, where
 $s\in K^{\times} \I_{\mathfrak{m}_{R}}$. 
By definition, an analytical isomorphism $\Upsilon:\C_{\infty}/\mathfrak{a}_{R}\rightarrow \D = (\C_{\infty},\uprho )$ of $R$-modules takes the analytical torsion 
\[ {\rm Tor}_{\mathfrak{a}_{R}}(\mathfrak{m}_{R}):= (\mathfrak{m}_{R}\mathfrak{a}_{R}^{-1})^{\ast}/\mathfrak{a}_{R}\] to the algebraic torsion $\D[\mathfrak{m}_{R}]$.  
We will need to identify the image of ${\rm Tor}_{\mathfrak{a}_{R}}(\mathfrak{m}_{R})$ under the isomorphism 
\begin{align}\label{ThetaDef} \Uptheta: K/\mathfrak{a}_{R}\cong   \bigoplus_{\mathfrak{p}\subset A} K_{\mathfrak{p}}/ \mathfrak{a}_{R,\mathfrak{p}}\end{align}
of Theorem \ref{AfModIso}.  If the factorization in $A$ of $\mathfrak{m}$ is $\mathfrak{m}=\prod \mathfrak{p}^{n_{\mathfrak{p}}}$,  then $\mathfrak{m}_{R}$ has the primary
factorization
\[ \mathfrak{m}_{R}=\bigcap  (\mathfrak{p}^{n_{\mathfrak{p}}})_{R}   .\]
We will be interested in the
torsion groups corresponding to the $\mathfrak{p}$-primary factors,  \[ {\rm Tor}_{\mathfrak{a}_{R}}(  (\mathfrak{p}^{n_{\mathfrak{p}}})_{R}),\]
which canonically include into the corresponding summand $K_{\mathfrak{p}}/ \mathfrak{a}_{R,\mathfrak{p}}$ on the right-hand side of (\ref{ThetaDef}).  Accordingly,
the notation ${\rm Tor}_{\mathfrak{a}_{R}}(  (\mathfrak{p}^{n_{\mathfrak{p}}})_{R})$ will also be used to denote the image in $K_{\mathfrak{p}}/ \mathfrak{a}_{R,\mathfrak{p}}$.
\begin{lemm}\label{DirSumTor} $\Uptheta$ takes ${\rm Tor}_{\mathfrak{a}_{R}}(\mathfrak{m}_{R})\subset  K/\mathfrak{a}_{R}$ isomorphically onto
\[\bigoplus_{\mathfrak{p}\supset\mathfrak{m}} {\rm Tor}_{\mathfrak{a}_{R}}(  (\mathfrak{p}^{n_{\mathfrak{p}}})_{R})\]
\end{lemm}

\begin{proof}  It is enough to observe that the summation-of-coordinates map, used in the proof of Theorem \ref{AfModIso}, defines an isomorphism 
\[ S:\bigoplus_{\mathfrak{p}\supset\mathfrak{m}} {\rm Tor}_{\mathfrak{a}_{R}}(  (\mathfrak{p}^{n_{\mathfrak{p}}})_{R}) \stackrel{\cong}{\longrightarrow} {\rm Tor}_{\mathfrak{a}_{R}}(\mathfrak{m}_{R}) ,
\quad \upmu = (\upmu_{\mathfrak{p}}) \longmapsto \sum\upmu_{\mathfrak{p}} .\]  
Indeed, since $\mathfrak{m}_{R}\subset (\mathfrak{p}^{n_{\mathfrak{p}}})_{R}$, 
\[ {\rm Tor}_{\mathfrak{a}_{R}}(  (\mathfrak{p}^{n_{\mathfrak{p}}})_{R}) \subset  {\rm Tor}_{\mathfrak{a}_{R}}(\mathfrak{m}_{R}),\]
so $S$ defines an inclusion into $ {\rm Tor}_{\mathfrak{a}_{R}}(\mathfrak{m}_{R})$.  Following the notation of the proof of surjectivity of $S$ in Theorem  \ref{AfModIso}: given $m\in  {\rm Tor}_{\mathfrak{a}_{R}}(\mathfrak{m}_{R})$, the unique element mapping onto it is given by 
$\upmu = (\upmu_{\mathfrak{p}})$, which has non-0 coordinates only at $\mathfrak{p}\supset\mathfrak{m}$, with the property that
$\upmu_{\mathfrak{p}}=m\upvarepsilon_{\mathfrak{p}}$, where $\sum \upvarepsilon_{\mathfrak{p}}=1$ and 
\[ \upvarepsilon_{\mathfrak{p}}\in  \bigcap_{\mathfrak{p} \not=\mathfrak{q}\supset \mathfrak{m}} (\mathfrak{q}^{n_{\mathfrak{q}}} )_{R}.\]
But then it follows that
\[  \upmu_{\mathfrak{p}}(\mathfrak{p}^{n_{\mathfrak{p}}})_{R}=   m\upvarepsilon_{\mathfrak{p}}   (\mathfrak{p}^{n_{\mathfrak{p}}})_{R}  \in m \mathfrak{m}_{R} \subset\mathfrak{a}_{R} ,\]
so $\upmu_{\mathfrak{p}} \in {\rm Tor}_{\mathfrak{a}_{R}}(  (\mathfrak{p}^{n_{\mathfrak{p}}})_{R})$, as asserted.  
\end{proof}


In what follows we will assume that $\mathfrak{m}\subset A$ is an extension of an ideal $\mathfrak{n}\subset R$ i.e. 
\[  A\mathfrak{n} = \mathfrak{m} .\]
In this connection, we recall that the set of contractions of ideals in $A$ to $R$ is in bijection with the set of extensions  of ideals of $R$ to $A$,
\[  C_{R} := \{ \mathfrak{m}_{R}\; : \;\; \mathfrak{m}\subset A\}  \longleftrightarrow  E_{A} := \{ A\mathfrak{n} \; :\;\; \mathfrak{n}\subset R\};\]
the bijection is given by the contraction and extension operations.  See Proposition 1.17 of \cite{AM}.  Note that
every ideal  $\mathfrak{m}\subset A$ contains an element of $E_{A}$: since $A\mathfrak{m}_{R}\subset \mathfrak{m}$.  Thus, at the level of associated ray class fields, we have
\[ K_{\mathfrak{m}} \subset K_{A\mathfrak{m}_{R}}.\] 

\begin{prop}\label{primarysalsoexentions} If $\mathfrak{m}\in E_{A}$ with prime factorization $\mathfrak{m} = \prod \mathfrak{p}^{n_{\mathfrak{p}}}$, then $\mathfrak{p}^{n_{\mathfrak{p}}}\in E_{A}$.
\end{prop}

\begin{proof}  For $\mathfrak{n}_{1},\mathfrak{n}_{2}\subset R$, $A(\mathfrak{n}_{1}\cap \mathfrak{n}_{2})\subset (A\mathfrak{n}_{1})\cap (A\mathfrak{n}_{2})$ (\cite{AM}, Exercise 1.18).  It follows that
\[ \mathfrak{m} = A\mathfrak{m}_{R} \subset \bigcap A(\mathfrak{p}^{n_{\mathfrak{p}}} )_{R} \subset \bigcap \mathfrak{p}^{n_{\mathfrak{p}}} =\mathfrak{m}.\]  Hence
\[ \bigcap A(\mathfrak{p}^{n_{\mathfrak{p}}} )_{R} = \prod \mathfrak{p}^{n_{\mathfrak{p}}} .\]
But for each $\mathfrak{p}$, $A(\mathfrak{p}^{n_{\mathfrak{p}}} )_{R}\subset \mathfrak{p}^{n_{\mathfrak{p}}} $,
so by uniqueness of prime factorization in $A$, we must have $A(\mathfrak{p}^{n_{\mathfrak{p}}} )_{R} = \mathfrak{p}^{n_{\mathfrak{p}}} $.
\end{proof}


\begin{lemm}\label{torlem} Suppose that $s\in \I^{R}$ is such that $s\mathfrak{a}_{R}=\mathfrak{a}_{R}$.  Fix $\mathfrak{m}\in E_{A}$.  If  $s\in \I_{\mathfrak{m}_{R}}$ then $s$ acts as the identity on 
${\rm Tor}_{\mathfrak{a}_{R}}(\mathfrak{m}_{R})$.   If in addition, $\mathfrak{m}$ is either prime to $\mathfrak{c}$ or contained in $\mathfrak{c}$, then the converse is true as well.
\end{lemm}

\begin{proof}  Recall (see the discussion preceding Lemma \ref{DirSumTor}) that we may identify \[  {\rm Tor}_{\mathfrak{a}_{R}}(  (\mathfrak{p}^{n_{\mathfrak{p}}})_{R})\hookrightarrow 
K_{\mathfrak{p}}/\mathfrak{a}_{R,\mathfrak{p}}.\]  The hypothesis $s\mathfrak{a}_{R} = \mathfrak{a}_{R}$ implies for all $\mathfrak{p}$ that $s_{\mathfrak{p}}\mathfrak{a}_{R,\mathfrak{p}} =\mathfrak{a}_{R,\mathfrak{p}} $, whence
$s_{\mathfrak{p}}$ acts on each factor module $K_{\mathfrak{p}}/ \mathfrak{a}_{R,\mathfrak{p}}$, and in particular, the question as to its effect on $ {\rm Tor}_{\mathfrak{a}_{R}}(  (\mathfrak{p}^{n_{\mathfrak{p}}})_{R})$
is well-posed.

 Suppose $s=(s_{\mathfrak{p}})\in \I_{\mathfrak{m}_{R}}
$. 
Then for all $\mathfrak{p}\subset A$ prime, we claim that 
\begin{align}\label{killclaim} (s_{\mathfrak{p}}-1) {\rm Tor}_{\mathfrak{a}_{R}}(  (\mathfrak{p}^{n_{\mathfrak{p}}})_{R})
 \subset   \mathfrak{a}_{R, \mathfrak{p}}.\end{align}
 Indeed,  $s_{\mathfrak{p}} \in U_{\mathfrak{p}}^{(n_{\mathfrak{p}})}\cap R^{\times}_{\mathfrak{p}}$
 implies \[ s_{\mathfrak{p}} - 1\;\in\; \mathfrak{p}^{n_{\mathfrak{p}}}_{\mathfrak{p}}\cap R_{\mathfrak{p}}=(\mathfrak{p}_{\mathfrak{p}}^{n_{\mathfrak{p}}})_{R_{\mathfrak{p}}} =
  (\mathfrak{p}^{n_{\mathfrak{p}}})_{R, \mathfrak{p}},\] 
  where $(\mathfrak{p}_{\mathfrak{p}}^{n_{\mathfrak{p}}})_{R_{\mathfrak{p}}} $ denotes the contraction of  $\mathfrak{p}_{\mathfrak{p}}^{n_{\mathfrak{p}}}$
 to $R_{\mathfrak{p}}$.
   As $R_{\mathfrak{p}}$ 
 is an open subring of $K_{\mathfrak{p}}$ (since it contains $\mathfrak{c}_{\mathfrak{p}}$), it follows that $(\mathfrak{p}^{n_{\mathfrak{p}}})_{R}$ is dense in 
 $(\mathfrak{p}^{n_{\mathfrak{p}}})_{R, \mathfrak{p}}$.
 But by definition, 
 \begin{align}\label{densekillclaim}   (\mathfrak{p}^{n_{\mathfrak{p}}})_{R}\cdot {\rm Tor}_{\mathfrak{a}_{R}}(  (\mathfrak{p}^{n_{\mathfrak{p}}})_{R}) \subset \mathfrak{a}_{R},\end{align}
and since  $s_{\mathfrak{p}}-1$ is approximated by elements of  $(\mathfrak{p}^{n_{\mathfrak{p}}})_{R}$, 
 the claim (\ref{killclaim}) follows from (\ref{densekillclaim}).
 In particular, $s_{\mathfrak{p}}$ acts as the identity on  ${\rm Tor}_{\mathfrak{a}_{R}}(  (\mathfrak{p}^{n_{\mathfrak{p}}})_{R})$;
 by Lemma \ref{DirSumTor}, it follows that $s$ acts as the identity on ${\rm Tor}_{\mathfrak{a}_{R}}(\mathfrak{m}_{R})$.
 
 Now suppose that either  $\mathfrak{m}$ is  prime to $\mathfrak{c}$ or contained in $\mathfrak{c}$; we will show the converse of the statement we have just proved.
Consider first the case where $\mathfrak{c}\supset\mathfrak{m}$ and
suppose that $\mathfrak{p} | \mathfrak{c}$.  In particular, $\mathfrak{p}|\mathfrak{m}$, and after localizing,
\[ \mathfrak{p}_{\mathfrak{p}}^{n_{\mathfrak{p}} }\subset \mathfrak{c}_{\mathfrak{p}} \subset R_{\mathfrak{p}} ,\]
hence
\[     U^{ ( n_{\mathfrak{p}})}_{\mathfrak{p}} = 1 + \mathfrak{p}_{\mathfrak{p}}^{n_{\mathfrak{p}} }\subset R_{\mathfrak{p}}^{\times}  \]
 Since $s$ is prime to $\mathfrak{c}$, that is,  $s\in \I^{R}$, $s_{\mathfrak{p}}\in R_{\mathfrak{p}}^{\times}$, 
and so what remains is to show that $s_{\mathfrak{p}}\in U^{ ( n_{\mathfrak{p}})}_{\mathfrak{p}}\subset R_{\mathfrak{p}}^{\times}$.  In this case $\mathfrak{p}\nmid \mathfrak{a}$ and thus  $\mathfrak{a}_{\mathfrak{p}} = A_{\mathfrak{p}}$
and  $\mathfrak{a}_{R,\mathfrak{p}} = R_{\mathfrak{p}}$.  Additionally, since $\mathfrak{p}^{n_{\mathfrak{p}}}_{\mathfrak{p}}\subset \mathfrak{c}_{\mathfrak{p}}$, we have $\mathfrak{p}^{n_{\mathfrak{p}}}_{R,\mathfrak{p}}=
\mathfrak{p}^{n_{\mathfrak{p}}}_{\mathfrak{p}}$ and is moreover principal, generated by $\uppi^{n_{\mathfrak{p}}}$ if $\mathfrak{p}_{\mathfrak{p}} = \uppi A_{\mathfrak{p}}$.  In particular, $\mathfrak{p}^{n_{\mathfrak{p}}}_{\mathfrak{p}}$ 
is now invertible as an $R_{\mathfrak{p}}$-ideal, with inverse 
$
\uppi^{-n_{\mathfrak{p}}} R_{\mathfrak{p}}$.
 It follows that 
we may identify $ {\rm Tor}_{\mathfrak{a}_{R}}(  (\mathfrak{p}^{n_{\mathfrak{p}}})_{R})
$ in $K_{\mathfrak{p}}/\mathfrak{a}_{R,\mathfrak{p}}  =K_{\mathfrak{p}}/R_{\mathfrak{p}} $ with 
$
 \uppi^{-n_{\mathfrak{p}}} R_{\mathfrak{p}}/R_{\mathfrak{p}}$.  Now
\begin{align}\label{containmentofstars} 
\uppi^{-n_{\mathfrak{p}}}R_{\mathfrak{p}} \subset \uppi^{-n_{\mathfrak{p}}}A_{\mathfrak{p}}= \mathfrak{p}_{\mathfrak{p}}^{-n_{\mathfrak{p}}} : \end{align}
From this we conclude that the hypothesis that $s$ acts as the identity on ${\rm Tor}_{\mathfrak{a}_{R}}(\mathfrak{m}_{R})$ implies, by Lemma \ref{DirSumTor},
\[  (s_{\mathfrak{p}}-1) 
  \uppi^{-n_{\mathfrak{p}}} R_{\mathfrak{p}}\subset \mathfrak{a}_{R,\mathfrak{p}} =R_{\mathfrak{p}} \]
which in turn gives, after multiplying the above inclusion by $A_{\mathfrak{p}}$,
\[  (s_{\mathfrak{p}}-1)\mathfrak{p}^{-n_{\mathfrak{p}}}_{\mathfrak{p}}\subset A_{\mathfrak{p}} \]
i.e. $s_{\mathfrak{p}}\in U^{ ( n_{\mathfrak{p}})}_{\mathfrak{p}}$.
Now assume that $\mathfrak{p}| \mathfrak{m}$ but $\mathfrak{p} \nmid \mathfrak{c}$.  Then
$A_{\mathfrak{p}}=R_{\mathfrak{p}}$ and hence $\mathfrak{a}_{R,\mathfrak{p}}=\mathfrak{a}_{\mathfrak{p}}=\mathfrak{p}_{\mathfrak{p}}^{m} $ for some $m\geq 0$.
Moreover, $\mathfrak{p}^{n_{\mathfrak{p}}}_{R,\mathfrak{p}} = \mathfrak{p}^{n_{\mathfrak{p}}}_{\mathfrak{p}}$ is invertible
as an $R_{\mathfrak{p}}$ ideal, which allows us to identify 
${\rm Tor}_{\mathfrak{a}_{R}}(  (\mathfrak{p}^{n_{\mathfrak{p}}})_{R})$  with $\mathfrak{p}_{\mathfrak{p}}^{-n_{\mathfrak{p}}}/ \mathfrak{p}_{\mathfrak{p}}^{m} $.
Since, by hypothesis, $s_{\mathfrak{p}}$ acts as the identity on the latter, $s_{\mathfrak{p}} -1\in \mathfrak{p}_{\mathfrak{p}}^{m+n_{\mathfrak{p}}}  $ which implies
$s_{\mathfrak{p}} \in U^{ ( m+n_{\mathfrak{p}})}_{\mathfrak{p}} \subset U^{ ( n_{\mathfrak{p}})}_{\mathfrak{p}}$.  
If $\mathfrak{p}\nmid \mathfrak{c}, \mathfrak{m}$, no restrictions are imposed (i.e.\ $n_{\mathfrak{p}} =0$), and since $s\in \I^{R}$, we have trivially $s_{\mathfrak{p}} \in R_{\mathfrak{p}}^{\times} = A_{\mathfrak{p}}^{\times} = 
U^{ ( 0)}_{\mathfrak{p},R}$.
 This finishes the proof of the converse for $\mathfrak{m}\subset\mathfrak{p}$.

Finally, we prove the converse when $(\mathfrak{m},\mathfrak{c})=1$.  
If $\mathfrak{p} |  \mathfrak{c}$ then $\mathfrak{p} \nmid \mathfrak{m}$ and $n_{\mathfrak{p}}=0$.   There is then nothing to show since, by the overall hypothesis that $s\in \I^{R}$, we have immediately 
$s_{\mathfrak{p}}\in R_{\mathfrak{p}}^{\times} = U^{ ( 0)}_{\mathfrak{p},R}$.  On the other hand, if  $\mathfrak{p} |  \mathfrak{m}$ then $\mathfrak{p} \nmid \mathfrak{c}$, the proof is the same as that
in the previous paragraph, where we were assuming $\mathfrak{m}\subset\mathfrak{c}$.
The case where $\mathfrak{p}\nmid \mathfrak{c}, \mathfrak{m}$ is also handled exactly as in the previous paragraph.  This finishes the proof of the converse statement when $(\mathfrak{m},\mathfrak{c})=1$.
\end{proof}

 Let $\D =(\C_{\infty},\uprho )$ be a Drinfeld module defined over the minimal field of definition $H_{R}$ (= the Hilbert class field) and let $e_{\uprho}$ be the exponential inducing an isomorphism
 \[   e_{\uprho}: \C_{\infty}/\Uplambda_{\uprho}\longrightarrow \D.  \]
 Then there exists $\mathfrak{a}\subset R$ an ideal and $\upxi\in \C_{\infty}$ so that $\Uplambda_{\uprho}= \upxi \mathfrak{a}$.  We point out that $\D$ is in general
 {\it not} sign normalized, unless $d_{\infty}=1$.

\begin{theo} Let $\D$ be as in the previous paragraph, $\mathfrak{m}\in E_{A}$ a modulus which is either prime to $\mathfrak{c}$ or contained in $\mathfrak{c}$. Then
\[  K_{\mathfrak{m}_{R}} = H_{R}\big(e_{\uprho} (\upxi t)|\; t\in {\rm Tor}_{\mathfrak{a}_{R}}(\mathfrak{m}_{R}) \big) .\]
\end{theo}

\begin{proof} By definition, $K_{\mathfrak{m}_{R}} $ is abelian. Moreover, 
by the order analog of Theorem 3.1.1 of \cite{Thak} (whose
proof is identical to that in  the case $R=A$), $H_{R}(e_{\uprho} (\upxi t)|\; t\in {\rm Tor}_{\mathfrak{a}_{R}}(mR ) )$ is abelian for any $m\in \mathfrak{m}_{R}$, hence
$H_{R}(e_{\uprho} (\upxi t)|\; t\in {\rm Tor}_{\mathfrak{a}_{R}}(\mathfrak{m}) )$ is abelian as well. 
  Therefore,
it will be enough to show that the two fields appearing in the statement of the Theorem are the fixed fields of the same subgroup of ${\rm Gal}(K^{\rm ab}_{\infty}/K)$.  Let $\upsigma = [s, K]$.  Then 
\[  \upsigma|_{K_{\mathfrak{m}_{R}}}\text{ is trivial }\Longleftrightarrow s^{-1} =\upalpha u, \;\; \upalpha\in K^{\times} \text{ and } u\in \I_{\mathfrak{m}_{R}}. \]
Suppose first that $  \upsigma|_{K_{\mathfrak{m}_{R}}}$ is trivial.  Then $\D^{\upsigma}=\D$ since $\D$ is defined over $H_{R}\subset K_{\mathfrak{m}_{R}} $.  Choose the analytical isomorphism 
$\Upsilon$ in the Main Theorem to be the composition $\mathfrak{a}\rightarrow \upxi\mathfrak{a}$
with the exponential $e_{\uprho}$.
By the Main Theorem, it follows that $s^{-1}\mathfrak{a}$ also parametrizes $\D$ and therefore is a multiple of $\mathfrak{a}$ by an element of $K^{\times}$.
In particular, we may choose the element $\upalpha$ above so that $s^{-1}\mathfrak{a}=\mathfrak{a}$.   But for all $\mathfrak{p}\subset A$, $u_{\mathfrak{p}}\in R^{\times}_{\mathfrak{p}}$ (by definition of $\I_{\mathfrak{m}_{R}}$),
hence $u_{\mathfrak{p}} \mathfrak{a}_{R,\mathfrak{p}}=\mathfrak{a}_{R,\mathfrak{p}}$, which implies $\upalpha \mathfrak{a}_{R,\mathfrak{p}}=\mathfrak{a}_{R,\mathfrak{p}}$, forcing $\upalpha\in \F_{q}^{\times}$.
We may thus choose $\upalpha=1$.  Then,
by Lemma \ref{torlem}, $s$ acts trivially on the analytic torsion ${\rm Tor}_{\mathfrak{a}_{R}}(\mathfrak{m}_{R})$, and by the Main Theorem, it follows that $\upsigma$
 fixes $H_{R}( e_{\uprho} (\upxi t)|\; t\in {\rm Tor}_{\mathfrak{a}_{R}}(\mathfrak{m}_{R}) ) $.  
  In the other direction, suppose $\upsigma=[s,K]$ fixes $H_{R}( e_{\uprho} (\upxi t)|\; t\in {\rm Tor}_{\mathfrak{a}_{R}}(\mathfrak{m}_{R}) ) $.  In particular, $\D^{\upsigma}=\D$,
and after re-choosing $\upalpha$, we may assume $s^{-1}\upalpha\mathfrak{a}_{R}=\mathfrak{a}_{R}$. The latter implies that $s^{-1}\upalpha\in \I^{R}$: indeed, it implies that for all $\mathfrak{p}\subset A$
$s_{\mathfrak{p}}^{-1} \upalpha\mathfrak{a}_{R,\mathfrak{p}} = \mathfrak{a}_{R,\mathfrak{p}}$.  If
 $\mathfrak{p}|\mathfrak{c}$, then
$\mathfrak{a}_{R,\mathfrak{p}} = R_{\mathfrak{p}}$, hence for such $\mathfrak{p}$, $s_{\mathfrak{p}}^{-1}\upalpha\in R_{\mathfrak{p}}^{\times}$, which implies $s^{-1}\upalpha\in \I^{R}$.
By the Main Theorem and its proof (especially (\ref{diagram2})), $\Upsilon' =\Upsilon\circ \upalpha^{-1}$ restricted to ${\rm Tor}_{\mathfrak{a}_{R}}(\mathfrak{m}_{R}) $, so 
\[  \Upsilon \circ s^{-1}\upalpha = \upsigma\circ\Upsilon=\Upsilon \quad \text{restricted to}\;\; {\rm Tor}_{\mathfrak{a}_{R}}(\mathfrak{m}_{R}).\]  From this
 we conclude $s^{-1}\upalpha$ acts trivially on ${\rm Tor}_{\mathfrak{a}_{R}}(\mathfrak{m}_{R})$ as well, 
and by Lemma
\ref{torlem}, (applied to $s^{-1}\upalpha$), we have
$s\in \I_{\mathfrak{m}_{R}}K^{\times}$.  Therefore, by (\ref{reciprocitymapforrcf}),  $\upsigma=[s,K]$ is the identity on $K_{\mathfrak{m}_{R}} $.
\end{proof}

\section{Generation of Hilbert Class Fields of Real Quadratic Rank 2 Orders}\label{HigherHCF}

In this section we will consider $K/\F_{q}[T]$ quadratic and real, where the latter means that over the place $\infty$ there are two places $\infty_{1},\infty_{2}$ in the associated curve
$\Upsigma_{K}$.  The Dedekind domain $A$ will be defined as the ring of functions regular outside of $\infty_{1}$.  The absolute value associated to $\infty$ is denoted $|\cdot |$.
 
Let $f\in k_{\infty}$ and define \[ \Uplambda_{\upvarepsilon}(f)=\{a\in  \F_{q}[T],\| af\|<\upvarepsilon\},\] where $\| x\|$ = the distance of $x$ to the nearest element of $ \F_{q}[T]$,
\[ \upzeta_{f,\upvarepsilon}(n) = \sum_{\upalpha\in \Uplambda_{\upvarepsilon}(f)\;\text{\rm monic}} \upalpha^{-n} ,\quad n\in\N.\]
Define
\begin{align*} \Updelta_{\upvarepsilon} (f) &:=-(T^{q^{2}}-T)\upzeta_{f,\upvarepsilon}(q^{2}-1)+ (T^{q}-T)^{q}\upzeta_{f,\upvarepsilon}(q-1)^{q+1}
 \end{align*} and 
\[g_{\upvarepsilon} (f) :=-(T^{q}-T)\upzeta_{f,\upvarepsilon}(q-1).\]
Then the  $\upvarepsilon$-modular invariant of $f$ is defined
\[ j_{\upvarepsilon}(f):=\frac{{g_{\upvarepsilon}}^{q+1}(f)}{\Updelta_{\upvarepsilon}(f)} = \frac{1}{\frac{1}{T^{q}-T}  -J_{\upvarepsilon}(f) } \]
where
\begin{align} J_{\upvarepsilon}(f) :=  \frac{T^{q^{2}}-T}{(T^{q}-T)^{q+1}} \cdot
\frac{\upzeta_{f,\e}(q^{2}-1)}{\upzeta_{f,\e}(q-1)^{q+1}}.
\end{align}
The {\bf {\em quantum modular invariant}} or {\bf {\em quantum j-invariant}}  of $f$ is 
\[  j^{\rm qt}(f) := \lim_{\e\rightarrow 0} j_{\upvarepsilon}(f)   \subset k_{\infty}\cup \{ \infty \},\]
where by $ \lim_{\e\rightarrow 0} j_{\upvarepsilon}(f) $ we mean the {\it set of limit points} of convergent sequences $\{ j_{\e_{i}}(f)\}$, $\e_{i}\rightarrow 0$.  
 
 We remark that the association
$f\mapsto j^{\rm qt}(f) $ is non-trivially multi-valued:  see \cite{DGII}.  It is invariant with respect to the action of ${\rm PGL}_{2}( \F_{q}[T])$ on $k_{\infty}$ (i.e.\
$j^{\rm qt}(Mf)=j^{\rm qt}(f)$ for all $M\in {\rm PGL}_{2}( \F_{q}[T])$ and $f\in k_{\infty}$, see \cite{DGI}), and so defines
a multi-valued function
\[ j^{\rm qt}(f): {\rm PGL}_{2}( \F_{q}[T])\backslash k_{\infty}\multimap k_{\infty}\cup \{\infty\}.\]

Let $K/\F_{q}(T)$ be real and quadratic.  In what follows, let $f_{0}\in \mathcal{O}_{K}^{\times}$ be a fundamental unit
and write $f=f_{0}^{k}$, $k\in \Z$. We define the quadratic order
\[ R =  A_{f}:=\F_{q}[f, fT,\dots ,fT^{d-1}]\subset A =\F_{q}[f_{0}, f_{0}T,\dots ,f_{0}T^{d_{0}-1} ],\]
   where $ |f_{0}| =q^{-d_{0}}$ and  $d=kd_{0}$, and where the equality $A =\F_{q}[f_{0}, f_{0}T,\dots ,f_{0}T^{d_{0}-1} ]$ is proved for example 
   in \cite{DGIII}.
 Thus, $A_{f_{0}} = A$. 

\begin{lemm}\label{basisLemmaAfk}  The set
\begin{align}\label{proposedbasis} \left\{ 1, f, fT,\dots , f T^{d-1}, f^{2}, f^{2}T, \dots ,   f^{2} T^{d-1}, \dots \right\}\end{align}
is an $\F_{q}$ vector space basis of $A_{f}$.  In particular, if $g\in A\setminus \F_{q}$ and 
$-d\leq \deg_{\infty_{1}} (g)$ then $g\not\in A_{f}$.
\end{lemm}

 \begin{proof} Let $G(X_{0}, \dots , X_{d-1}) \in \F_{q} [X_{0}, \dots , X_{d-1}]$ with 
\[ \upalpha = G(f, \dots , fT^{d-1})  \in A_{f}.\]
Then $\upalpha$ is a linear combination of monomials of the form
\[ f^{l} T^{m} ,\quad l,m\geq 0 \]
in which, if $l=0$,  then $m=0$.  Suppose first that $m= d$.  Then using the minimal polynomial for $f$, 
\begin{align}\label{MinPolyf} f^{2} = {\tt a}f+{\tt b},\quad {\tt a} \in \F_{q}[T], \;\deg_{T}({\tt a}) = d,\;\; {\tt b}\in \F_{q}^{\times},\end{align}
we may solve
for $T^{d}$ as 
\[  T^{d} =  f-{\tt a}_{\rm lower} -\frac{\tt b}{f}
\]
where 
\[ {\tt a}_{\rm lower}  = {\tt a}- \text{leading term} = {\tt a}-T^{d} .\]
We then have
\[ f^{l}T^{d} = f^{l}\left(  f -{\tt a}_{\rm lower} -\frac{\tt b}{f} \right) \]
which is a linear combination of elements in (\ref{proposedbasis}).
By induction, one can express any monomial of the shape $ f^{l} T^{m}$ as a linear combination of monomials occurring in the set (\ref{proposedbasis}).
\end{proof}

Consider the ideals 
\begin{align}\label{aidef} \mathfrak{a}_{i} =\mathfrak{a}_{f,i} =(f, fT,\dots ,fT^{i})\subset A_{f}, \quad i=0,1,\dots ,d-1 .\end{align}

 \begin{prop}\label{PowerProp} For all $i$, $\mathfrak{a}_{i}= \mathfrak{a}_{d-1}^{d-i}$.  In particular, $Z = \{ \mathfrak{a}_{0},\dots , 
 \mathfrak{a}_{d-1}\}$ forms a cyclic subgroup of the class group ${\sf Cl}_{A_{f}}$.
      \end{prop}
      
   \begin{proof} See also \cite{DGIII}, where a proof is given for ${\tt a}=T^{d}$.  A generating set for $\mathfrak{a}^{2}_{d-1}$ is $\{ f^{2}T^{j}\}_{j=0}^{2d-2}$.  Now the  generators with $j\geq d$
may be written
   \begin{align}\label{genrewrite} f^{2}T^{d} =f^{3}- f^{2}{\tt a}_{\rm lower}-f{\tt b},\;\;\dots\; , f^{2}T^{2d-2} =T^{d-2} ( f^{3}- f^{2}{\tt a}_{\rm lower}-f{\tt b}),\end{align}
        where, as in Lemma \ref{basisLemmaAfk}, ${\tt a}_{\rm lower}= {\tt a}-T^{d}$.  It follows by inspection that these generators all belong to $\mathfrak{a}_{d-2}$.  The generators with $j\leq d-1$ trivially
        belong to $\mathfrak{a}_{d-2}$, which gives the inclusion $\mathfrak{a}^{2}_{d-1}\subset \mathfrak{a}_{d-2}$.  On the other hand,
        \[  f^{3} -f^{2}{\tt a}_{\rm lower}, \dots ,T^{d-2}(  f^{3} -f^{2}{\tt a}_{\rm lower}) \in \mathfrak{a}^{2}_{d-1},  \]
and together with (\ref{genrewrite}), this gives $f,\dots ,fT^{d-2}\in  \mathfrak{a}^{2}_{d-1}$ i.e.\ $\mathfrak{a}^{2}_{d-1}= \mathfrak{a}_{d-2}$.  The general statement proceeds by showing, by induction,
that $\mathfrak{a}_{d-i}\mathfrak{a}_{d-1}=\mathfrak{a}_{d-i-1}$.
      \end{proof}

By Proposition \ref{PowerProp}, the $\mathfrak{a}_{i}$ are invertible and thus define elements of ${\sf Cl}_{A_{f}}$.  For any $\mathfrak{a}\subset A_{f}$,
the modular invariant $j(\mathfrak{a})$ was defined in (\ref{defnofj}) of section \S \ref{RCFofR}.  

Now consider the allied ``rank 2'' order 
 \[ \mathcal{O}_{f} =\F_{q} [T] [f,f^{-1}]= \F_{q} [T, f, f^{-1}] \subset \mathcal{O}_{K}.\]
 It has conductor $\mathfrak{C}$ for which $A\cap \mathfrak{C}=\mathfrak{c}$ = the conductor of $A_{f}\subset A$.
Its Hilbert class field $H_{\mathcal{O}_{f}}$  is defined via class field theory as the abelian extension associated to
the group
\[ C_{\mathcal{O}_{f}} =\frac{ \I_{\mathcal{O}_{f}} \cdot K^{\times}}{K^{\times}}  \]
where
\[   \I_{\mathcal{O}_{f}} = K^{\times}_{\infty_{1}}\cdot K^{\times}_{\infty_{2}}\cdot \prod_{\mathfrak{p}\not=\infty_{1},\infty_{2}}  (\mathcal{O}_{f})^{\times}_{\mathfrak{p}} .\]  
We have the following diamond of Hilbert class fields
\[ 
\begin{diagram}
H_{A_{f}} &  & \\
\vLine  & \rdLine & \\
H_{A}  && H_{\mathcal{O}_{f}} \\
  & \rdLine &\vLine \\
& &H_{\mathcal{O}_{K}} . 
\end{diagram}
\]
As usual,

The purpose of the present section is to prove the following Theorem, which is the order
analog of the main result of \cite{DGIII}:

\begin{theo}\label{OrderGenTheorem} The Hilbert class field $H_{\mathcal{O}_{f}}$ is primitively generated over $K$ by
\[    \prod_{\upalpha\in j^{\rm qt}(f)} \upalpha = \prod_{i=0}^{d-1} j(\mathfrak{a}_{i}) .   \]
\end{theo}

 All of the arguments used in  \cite{DGIII}  extend in an expected way
to $H_{A_{f}}$, particularly because the main ingredient in most of our proofs was that $f$ is a unit, not necessarily fundamental. The assumption that $f$ is a fundamental unit
makes an explicit appearance in the middle of \S 1 of \cite{DGIII}: this was done so that we could identify generators of $A$, which necessarily use a fundamental unit.  In 
what follows, we will review carefully all of the elements in the proof appearing
in \cite{DGIII}, pointing out whatever changes need to be made to deal with the fact that $A_{f}$ is no longer a Dedekind domain.
The main point is that 
\[ A_{f}= \F_{q}[ f, fT,\dots , fT^{d-1}]\]
has exactly the same shape as $A$ when expressed explicitly using generators.

\vspace{3mm}

\noindent \S 1 of \cite{DGIII}: ``Diophantine Approximations of Quadratic Units''

\vspace{3mm}

 In the first part of the section, before assuming $f$ fundamental, we provide
an explicit formula for $\Uplambda_{\upvarepsilon}(f)$ which  we describe this now.
 Define 
${\tt Q}_{n}\in \F_{q}[T]$ by 
$ {\tt Q}_{0}=1, {\tt Q}_{1}={\tt a},\dots , {\tt Q}_{n+1}={\tt a}{\tt Q}_{n} +{\tt b}{\tt Q}_{n-1}$,
${\tt a},{\tt b}$ are as in (\ref{MinPolyf}).  
If $f^{\ast}$ denotes the Galois conjugate of $f$, 
we may assume $|f|>|f^{\ast}|$, and then
 $|f|=|{\tt a}|=q^{d}$ and $|f^{\ast}|=q^{-d}$.  Let ${\tt D}={\tt a}^{2}+4{\tt b}$ be the discriminant.
Using Binet's formula
\begin{align}\label{Binetform}{\tt Q}_{n} =\frac{ f^{n+1}-(f^{\ast})^{n+1}}{\sqrt{{\tt D}}},\quad n=0,1,\dots ,\end{align}
one may show that
$\| {\tt Q}_{n}f\| =  q^{ -(n+1) d } $,
from which it follows that the set \begin{align*}
\mathcal{B}=\{ T^{d-1}{\tt Q}_{0}, \dots , T{\tt Q}_{0}, {\tt Q}_{0} ; T^{d-1}{\tt Q}_{1}, \dots , T{\tt Q}_{1}, {\tt Q}_{1};\dots  \}\end{align*}
forms an $\F_{q}$ basis of $\F_{q}[T]$.
Write 
$\mathcal{B}(i) = \{ T^{d-1}{\tt Q}_{i},\dots , {\tt Q}_{i}\}$
 for the $i$th block of $\mathcal{B}$ and for $0\leq \tilde{d}\leq d-1$, denote
$ \mathcal{B}(i)_{\tilde{d}} = \{T^{\tilde{d}}{\tt Q}_{i},\dots , {\tt Q}_{i}\}$.
Then (see \cite{DGIII}, Lemma 1) 
\[  \Uplambda_{q^{-Nd-l}}(f )={\rm span}_{\F_{q}} ( \mathcal{B}(N)_{d-1-l},\mathcal{B}(N+1),\dots   ) .\]

Then, there is the renormalization result, Proposition 1 of  \cite{DGIII}, which shows that the renormalized $\Uplambda_{\upvarepsilon}(f)$ converge to the $\mathfrak{a}_{i}$.
The proof of this statement is valid for an arbitrary order, since it only uses Binet's formula.

\vspace{3mm}

\noindent \S 2 of \cite{DGIII}: ``The Quantum $j$-Invariant in Positive Characteristic''

\vspace{3mm}

The definitions of the $j$ invariant for ideal classes of $A_{f}$ and the associated quantum $j$-invariant are the same as those that appear in \cite{DGIII}.  The analog of Lemma 2 of \cite{DGIII}
is proved in Proposition \ref{PowerProp}.  Finally, Theorem 4 of \cite{DGIII}, which says that
\[ j^{\rm qt}(f) = \{ j( \mathfrak{a}_{i} )\}_{i=0}^{d-1}\]
remains valid for an arbitrary unit $f$, since the proof only uses Binet's formula.

\vspace{3mm}

\noindent \S 3 of \cite{DGIII}: ``Injectivity''

\noindent

\vspace{3mm}

 In the very beginning of this section, we pick an ideal class $[\mathfrak{a}]\in {\sf Cl}_{A}$ and write $\mathfrak{a}=(g,h)$.
In the case of $A_{f}$, by definition, an element of $[\mathfrak{a}]\in {\sf Cl}_{A_{f}}$ is the class of an invertible ideal.  
We prove above (see Theorem \ref{2gentheo})  that invertible ideals may be generated
by two elements.  So this presentation of $\mathfrak{a}=(g,h)$ remains valid.  We also need that $g$ has the smallest degree amongst all non-0 elements of $\mathfrak{a}$.  We may make
this assumption in view of {\it Note} \ref{mingennote} of section \S \ref{s:A}. The remainder of this section, which comprises the main technical results of \cite{DGIII}, only uses the fact that $f$ is a unit and the shape of $A_{f}$
as displayed above, which is identical in form to that of $A$.  Therefore these results hold identically in the case of $A_{f}$.

\vspace{3mm}

\noindent 

\vspace{3mm}

 \noindent \S 4 of \cite{DGIII}: ``Generation of the Hilbert Class Field'' 

\vspace{3mm}  Just as in the case of $A$, since $K$ is totally real, 
$H^{1}_{A_{f}}=$ the narrow Hilbert class field  of $A_{f}$ $=H_{A_{f}}$. The analog of Theorem 7 of \cite{DGIII} for $A_{f}$ is simply Corollary \ref{cor} of section \S \ref{RCFofR}, we simply note that $\mathfrak{a}\subset K$ is now an invertible $A_{f}$ fractional ideal.  The analog of Proposition 2 of \cite{DGIII} is the following:
 \begin{prop}\label{ZProp}  Let $  \mathfrak{a}_{d-1}\in  {\sf Cl}_{A_{f}} $ be as above.  Then 
       $ \langle  \mathfrak{a}_{d-1}\rangle  = {\rm Ker}( {\sf Cl}_{A_{f}} \rightarrow {\sf Cl}_{\mathcal{O}_{f}}  )$. In particular, the Galois group
      $  Z={\rm Gal}(H_{A_{f}}/H_{\mathcal{O}_{f}})$ is canonically isomorphic to $ \langle \mathfrak{a}_{d-1}\rangle $ via Class Field Theory.
       \end{prop}
       
       \begin{proof} For $f=f_{0}$ a fundamental unit, this result was proved in Proposition 2 of  \cite{DGIII}; we will use the latter to prove the Proposition. Denote
       $ Z_{0} :=$ ${\rm Gal}(H_{A}/H_{\mathcal{O}_{K}})$  $\cong$ ${\rm Ker}( {\sf Cl}_{A} \rightarrow {\sf Cl}_{\mathcal{O}_{K}}  )$ $= \langle \mathfrak{a}_{d_{0}-1}\rangle$, where
         $  \mathfrak{a}_{d_{0}-1} = (f_{0}, f_{0}T,\dots , f_{0}T^{d_{0}-1} ) $.
         In what follows, we simply write $Z={\rm Ker}( {\sf Cl}_{A_{f}} \rightarrow {\sf Cl}_{\mathcal{O}_{f}}  )$ and $Z_{0}={\rm Ker}( {\sf Cl}_{A} \rightarrow {\sf Cl}_{\mathcal{O}_{K}}  ) =
         \langle
         \mathfrak{a}_{d_{0}-1}\rangle$.
       We have the following commutative diagram of short exact sequences:
       \[
       \begin{diagram}
   1 &\rTo &  Z & \rInto &  {\sf Cl}_{A_{f}} \cong {\sf I}_{A_{f}}^{\mathfrak{c}} / {\sf P}_{A_{f}}^{\mathfrak{c}} & \rOnto^{\Upphi} &    
       {\sf Cl}_{\mathcal{O}_{f}} \cong {\sf I}_{\mathcal{O}_{f}}^{\mathfrak{C}} / {\sf P}_{\mathcal{O}_{f}}^{\mathfrak{C}} &\rTo &1 \\ 
 & &   \dTo & &   \dTo_{\Uppsi_{A}} & & \dTo_{\Uppsi_{\mathcal{O}}} & & \\
  1 & \rTo&  Z_{0}=  \langle
         \mathfrak{a}_{d_{0}-1}\rangle &\rInto &  {\sf Cl}_{A} \cong  {\sf I}_{A} / {\sf P}_{A}   & \rOnto_{\Upphi_{0}} & {\sf Cl}_{\mathcal{O}_{K}} \cong  {\sf I}_{\mathcal{O}_{K}} / {\sf P}_{\mathcal{O}_{K}} & \rTo& 1,
        \end{diagram}\]
        in which $ {\sf Cl}_{A_{f}} \cong {\sf I}_{A_{f}}^{\mathfrak{c}} / {\sf P}_{A_{f}}^{\mathfrak{c}} $ follows from Theorem \ref{CGIsos}.  The identification
        ${\sf Cl}_{\mathcal{O}_{K}} \cong  {\sf I}_{\mathcal{O}_{K}} / {\sf P}_{\mathcal{O}_{K}} $ is proved as in Theorem \ref{CGIsos}, where we remark that the proofs of the supporting results go thorugh
        in rank 2 as they are rank independent.  The vertical maps $\Uppsi_{A}, \Uppsi_{\mathcal{O}}$ are induced by ideal expansion, e.g., $ \mathfrak{b}_{f} \longmapsto \mathfrak{b}:=\mathfrak{b}_{f}A_{\infty_{1}}$, 
        which
        are bijections on ideals prime to $\mathfrak{c}$ (see Theorem \ref{bijectionprimetoc}).  
        We have the inclusion 
       $ \langle \mathfrak{a}_{d-1}\rangle\subset Z$ since $ \mathfrak{a}_{d-1}$ contains the $\mathcal{O}_{f}$-unit $f$.  Suppose there is an ideal class $\mathfrak{b}_{f}\in Z\setminus \langle \mathfrak{a}_{d-1}\rangle$;
       without loss of generality,  by Lemma \ref{relprimerep}, we may assume that $\mathfrak{b}_{f}$ is prime to $\mathfrak{a}_{d-1}$.    Then by Theorems  \ref{inverttheo} and \ref{invertiffprimetoc},
        $\mathfrak{b}_{f}$ may be written as a product of primes  that are prime to the conductor.  We claim that 
        $\mathfrak{a}_{d-1}A \subset \mathfrak{a}_{d_{0}-1}$.   Indeed, the generators of $\mathfrak{a}_{d-1}$ are of the form $fT^{j} $, $j=d_{0}m+i\in \{ 0,\dots ,d-1= nd_{0}-1\}$; writing the exponent of $T$ as 
        $j = md_{0} + i$ for $m<n$ and $i<d_{0}$, we
    have
        $fT^{j} = f_{0}^{n}T^{j}  = ( f_{0} T^{d_{0}}) ^{m}\cdot f_{0}^{n-m} T^{i} $.
        Using the quadratic relation for $f_{0}$, $f_{0}T^{d_{0}} \in \mathfrak{a}_{d_{0}-1}$, and this proves the claim.  But by
         Theorem \ref{bijectionprimetoc}, the expansion $\mathfrak{a}_{d-1}A$ is prime, hence  $\mathfrak{a}_{d-1}A=\mathfrak{a}_{d_{0}-1}$.
        In particular,  the expansion  $\mathfrak{b}=\mathfrak{b}_{f}A$ is therefore a product of primes that do not involve $\mathfrak{a}_{d_{0}-1}$.  By the commutativity
        of the diagram, $\mathfrak{b}\in Z_{0} = \langle \mathfrak{a}_{d_{0}-1}\rangle$, contradiction.
        \end{proof}

   The rest of the proof of Theorem \ref{OrderGenTheorem} 
follows {\it mutatis mutandis} that of \cite{DGIII}.



\end{document}